\documentclass[a4paper,12pt,dvipdfmx]{amsart}
\usepackage{amsmath}
\usepackage{amssymb}
\usepackage{enumerate}
\usepackage{amscd}
\usepackage{graphics}
\usepackage{latexsym}
\usepackage{verbatim} 
\usepackage{url}

\theoremstyle{plain}
\newtheorem{thm}{{\bf Theorem}}[section]
\newtheorem{cor}[thm]{{\bf  Corollary}}
\newtheorem{prop}[thm]{{\bf Proposition}}
\newtheorem{lemma}[thm]{{\bf Lemma}}
\newtheorem{fact}[thm]{{\bf Fact}}

\newtheorem{claim}[thm]{{\bf Claim}}

\theoremstyle{definition}
\newtheorem{define}[thm]{{\bf Definition}}
\newtheorem{note}[thm]{{\bf Note}}

\newcommand{\cf}{\mathord{\mathrm{cf}}}
\newcommand{\cof}{\mathord{\mathrm{cof}}}
\newcommand{\dom}{\mathord{\mathrm{dom}}}
\newcommand{\size}[1]{\left\vert {#1} \right\vert}
\newcommand{\p}{\mathcal{P}}
\newcommand{\ot}{\mathord{\mathrm{ot}}}

\newcommand{\col}{\mathord{\mathrm{Col}}}

\newcommand{\seq}[1]{\langle {#1} \rangle}

\newcommand{\ka}{\kappa}
\newcommand{\la}{\lambda}
\newcommand{\om}{\omega}

\newcommand{\bbP}{\mathbb{P}}
\newcommand{\bbQ}{\mathbb{Q}}
\newcommand{\bbR}{\mathbb{R}}

\newcommand{\calF}{\mathcal{F}}

\newcommand{\calH}{\mathcal{H}}

\newcommand{\HOD}{\mathrm{HOD}}

\newcommand{\gHOD}{\mathrm{gHOD}}
\newcommand{\M}{\mathbb{M}}

\title[The DDG and very large cardinals]{The downward directed grounds hypothesis and
very large cardinals}
\author[T. Usuba]{Toshimichi Usuba}
\address[T. Usuba]
{Faculty of Science and Engineering,
Waseda University, 
Okubo 3-4-1, Shinjyuku, Tokyo, 169-8555 Japan}
\email{usuba@waseda.jp}
\keywords{Forcing method, Set-theoretic geology, Downward directed grounds hypothesis, Large cardinal, Generic multiverse}
\subjclass[2010]{Primary 03E40, 03E45, 03E55}

\begin{document}

\begin{abstract}
A transitive model $M$ of ZFC is called a
ground if the universe $V$ is a set forcing extension of $M$.
We show that the grounds of $V$ are downward set-directed.
Consequently, we establish some fundamental theorems
on the forcing method and the set-theoretic geology. For instance,
(1) the mantle, the intersection of all grounds,
must be a model of ZFC.
(2) $V$ has only set many grounds if and only if
the mantle is a ground.
We also show that if the universe has some very large cardinal,
then the mantle must be a ground.
\end{abstract}
\maketitle

\section{Introduction}
After Cohen's solution of the continuum hypothesis,
{\it the forcing method} became an important tool of set-theory.
While many independent results are proved by the forcing method,
the nature of forcing is also a fundamental topic in this field. 

The forcing method uses second order objects, ground models, and generic extensions.
A natural framework for treating such objects is a second-order set theory such as
Von Neumann-Bernays-G\"odel set-theory.
However in the first order set theory ZFC,
it is more difficult to treat.
On the other hand, Laver \cite{Laver2} and, independently, Woodin
found that a ground model is definable in its set-forcing extension.
Later, Fuchs-Hamkins-Reitz \cite{FHR}
refined Laver and Woodin's result, and
they give a uniform definition of all ground models of the universe:
\begin{thm}[Reitz \cite{Reitz}, Fuchs-Hamkins-Reitz \cite{FHR}]
There is a first order formula $\varphi(x,r)$ of set theory such that:
\begin{enumerate}
\item For each $r$, the class $W_r=\{x: \varphi(x,r)\}$
is a ground model of $V$.
\item For every ground model $M$  of $V$,
there is $r$ such that $M=W_r$.
\end{enumerate}
\end{thm}
This allows us to treat 
ground models within the first order set theory ZFC.
Using this uniform definition, 
a study of the structure of all ground models was initiated,
which attempts a study of the nature of the forcing method.
The study is now called the {\it set-theoretic geology}.
In \cite{FHR}, some notions
in the forcing method and the set-theoretic geology were introduced.
One of the most important and natural concepts is the notion of the {\it mantle}:
\begin{define}
The \emph{mantle} is the intersection
of all ground models of the universe.
\end{define}
The mantle is a very natural and  canonical object in the context of the forcing method.
Fuchs-Hamkins-Reitz \cite{FHR} proved the structure of the mantle can be manipulated by
class forcings. For example,
they proved that the universe $V$ can be the mantle of some class forcing extension $V[G]$.
Besides the natural definition of the mantle,
there are many open questions about the mantle.
An important one is whether the mantle is a model of ZFC or not.
If it is a model of ZFC, then the mantle can be seen as the {\it core} of 
all ground models.

The key to the solution of this question is the downward directedness of the ground models.
The \emph{downward directed grounds hypothesis}, DDG,
is the assertion that every two ground models have a common ground model.
The \emph{strong DDG} is the assertion that 
every collection $\{W_r:r \in X\}$ of ground models indexed by some set $X$ has a common ground model.
In \cite{FHR}, it was proved that if DDG holds,
then the mantle is a model of ZF. Furthermore, if the strong DDG holds,
then the mantle satisfies the axiom of choice.
The DDG is not only a key to the solution, but it expresses a fundamental property of the structure of  ground models.

All known models, such as $L[A]$, $K$, $\HOD$, and class forcing extensions, satisfy
the strong DDG, and their mantles are  models of ZFC.
So one might expect that the DDG is a theorem of ZFC.
Meanwhile, attempts have been made to construct a counterexample by
the forcing method or inner model construction. 

In this paper,
we prove that the strong DDG is actually a theorem of ZFC.
\begin{thm}
The strong DDG holds.
\end{thm}
Because of this result, 
under ZFC, some fundamental properties of the 
set-theoretic geology can be derived without any assumptions. 
For example,  we prove the following:
\begin{thm}
\begin{enumerate}
\item The mantle  is a model of ZFC.
\item The universe $V$ has only set many ground models if and only if
the mantle is the minimum ground of all forcing extensions of $V$.
\end{enumerate}
\end{thm}

Next we look at the statement that
``$V$ has class many ground models''.
A natural model of this statement is given by a class Easton forcing extension,
and one can show that this statement is consistent with 
supercompact cardinals as follows.
For a cardinal $\ka$, 
by restricting the domain of Easton forcing,
we may assume that a class Easton forcing is $\ka$-directed closed.
Combining this observation with 
Laver's indestructiblity theorem (\cite{Laver}),
we have the consistency of
the statement that ``$V$ has class many ground models''+
``there exists a supercompact cardinal''.

In contrast to supercompact cardinals,
it is known that some very large cardinals, such as
extendible cardinals, cannot be indestructible for directed closed forcing
(Bgaria-Hamkins-Tsaprounis-Usuba \cite{BHTU}),
so we cannot prove the consistency 
of ``$V$ has class many ground models''+``there exists an extendible cardinal''
by this method.
Hence it is natural to ask whether existence of
class many ground models  is consistent with very large cardinals.
In the last half of this paper, we will give a partial answer to this question,
we  show that some very large cardinal is inconsistent with
class many ground models.

%We introduce the following new large cardinal.
\begin{define}
Let $\ka$ be an infinite cardinal.
We say that $\ka$ is  \emph{hyper-huge}
if for every $\la>\ka$,
there is an elementary embedding $j:V \to M$ into some inner model $M$
such that the critical point of $j$ is $\ka$,
$\la<j(\ka)$, and ${}^{j(\la)} M \subseteq M$.
\end{define}
% Clearly every hyper-huge cardinal
% is superhuge and extendible.
We show that our hyper-huge cardinal excludes
class many ground models. Furthermore, if a hyper-huge cardinal exists,
then the mantle must be  the minimum ground model
of the universe.
This indicates an unexpected connection between large cardinal axioms and the forcing method.
\begin{thm}
Suppose a hyper-huge cardinal exists.
Then the universe $V$ has only set many ground models,
and the mantle is a ground model of $V$.
\end{thm}

In the last section, we show some interesting consequences of
main theorems, and will prove some fundamental properties about 
set-theoretic geology and the generic multiverse.

\section{Basic materials}
In this section,
we present some basic notations, definitions, and known facts.
We also make some easy observations.% which will be used later.

Throughout this paper,
$V$ will denote the universe, so $V$ is a transitive model of
ZFC containing all ordinals unless otherwise specified.
\begin{note}
In this paper,
a \emph{class} 
means a second order object in the sense of
Von Neumann-Bernays-G\"odel set-theory
unless otherwise specified.
We do not require that a class $M$ is definable in $V$ with some
parameters, but
we assume that $V$ satisfies the replacement schemes 
for the formulas of the language $\{\in, M\}$
(where we identify $M$ as a unary predicate).
Note that, for every definable class $M$,
$V$ satisfies the replacement schemes 
for the formulas of the language $\{\in, M\}$.
Hence, every definable class is a \emph{class} in our sense.
Also note that if $W$ is a ground model of $V$,
then $W$ is a class in our sense.
\end{note}

In this paper, we will assume that a transitive model $M \subseteq V$ of ZF(C)
is a class of $V$ unless otherwise specified.

See  Theorem 13.9 in Jech \cite{Jech},
and see Definition 13.2 in \cite{Jech} for the definition of
the G\"odel operations.
\begin{fact}\label{model}
Let $N \subseteq V$ be a transitive class containing all ordinals.
Then $N$ is a model of ZF if and only if
$N$ is closed under the G\"odel operations 
and
$N$ is almost universal, that is, for every set $x \subseteq N$, there is
$y \in N$ with $x \subseteq y$.
\end{fact}
Note that the replacement schemes for the formulas of the language $\{\in, N\}$
is needed to prove this fact.

By Fact \ref{model}, for a transitive class $N \subseteq V$ containing all ordinals,
the statement ``$N$ is a model of ZFC'' is an  abbreviation of
the first order formula of the language $\{\in, N\}$ that
``$N$ is closed under the G\"odel operations''$+$
``$N$ is almost universal''$+$
``$N$ satisfies the axiom of choice''.

For a transitive model $M$ of ZFC and an ordinal $\alpha$,
let $M_\alpha$ be the set of all $x \in M$ with rank $<\alpha$.

We discuss the consistency strength of a hyper-huge cardinal.
Recall that an infinite cardinal $\ka$ is \emph{2-huge} if
there exists a transitive model $N$ of ZFC and an elementary embedding $j:V \to N$
with critical point $\ka$ such that $N$ is closed under $j(j(\ka))$-sequences.
\begin{lemma}
Suppose $\ka$ is 2-huge.
Then $V_\ka$ is a model of ZFC+``there are proper class many hyper-huge cardinals''.
Hence the consistency strength of the existence of a hyper-huge cardinal
is strictly weaker than of the existence of a 2-huge cardinal.
\end{lemma}
\begin{proof}
Let $j:V \to N$ be a 2-huge embedding with critical point $\ka$.
First we check that $V_{j(\ka)}$ is a model of ``$\ka$ is hyper-huge''.
To see this, take $\la$ with $\ka<\lambda<j(\ka)$.

Now we let $(\ka, \la) \to (\ka',\la')$ be the assertion that
there is an elementary embedding $j:V \to M$ for some $M$ such that:
\begin{enumerate}
\item the critical point of $j$ is $\ka$,
\item  $\la<j(\ka)=\ka'$,
\item $\la'=j(\la)$, and 
\item $M$ is closed under $\la'$-sequences.
\end{enumerate}
Note that $(\ka, \la) \to (\ka',\la')$ is equivalent to
the existence of a normal ultrafilter $U$ over $\p(\la')$ with
$\{x \subseteq \la':x \cap \ka \in \ka, \ot(x \cap \ka')=\ka,
\ot(x)=\la\} \in U$.

Since $j$ is a 2-huge embedding and $\lambda<j(\ka)$,
we have $j(\lambda)<j(j(\ka))$ and $(\ka, \lambda) \to (j(\ka), j(\lambda))$.
Now $\ka$ is $j(\ka)$-supercompact in $V$,
hence $j(\ka)$ is $j(j(\ka))$-supercompact in $N$, actually in $V$.
$j(\ka)<j(\lambda)<j(j(\ka))$, hence 
we have there is some $\ka',\lambda'<j(\ka)$ with
$(\ka, \lambda) \to (\ka', \lambda')$ in $V$ by
a standard reflection argument using the $j(j(\ka))$-supercompactness of $j(\ka)$.
This shows that $\ka$ is hyper-huge in $V_{j(\ka)}=N_{j(\ka)}$.
By the elementarity of $j$,
we have that $\{\alpha<\ka: \alpha$ is hyper-huge in $V_{\ka}\}$ is unbounded in $\ka$.
\end{proof}

A poset $\bbP$ is \emph{weakly homogeneous}
if, for every $p,q \in \bbP$,
there is an automorphism $f$ on $\bbP$ such that
$f(p)$ is compatible with $q$.

% For $Y \subseteq X$, $\col(Y)$ is a complete suborder of $X$.
% If $G$ is $(V, \col(X))$-generic, then let $G \restriction Y=G \cap \col(Y)$
% be the $(V, \col(Y))$-generic induced by $G$.

For an infinite cardinal $\chi$,
let $\calH(\chi)$ be the set of all sets $x$ such that
the transitive closure of $x$ has cardinality $<\chi$.
It is known that if $\chi$ is regular uncountable, then 
$\calH(\chi)$ is a transitive model of ZFC$-$P,
ZFC minus the power set axiom, with $\chi \subseteq \calH(\chi)$.

% \begin{note}\label{2.7.111}
% We proved Fact \ref{Kurepa}, Lemmas \ref{2.4.111} and \ref{2.5.111}
% as theorems of ZFC.
% But full ZFC is not necessary,
% it can be proved by ZFC-P.
% \end{note}

Here we recall some definitions and facts about set-theoretic geology.
See Fuchs-Hamkins-Reitz \cite{FHR} for more information.

\begin{define}%[Fuchs-Hamkins-Reitz \cite{FHR}]
A transitive  model $M \subseteq V$ of ZFC is a \emph{ground} of $V$ if
there exists a poset $\bbP \in M$ and
an $(M, \bbP)$-generic filter $G \subseteq \bbP$ with
$M[G]=V$.
\end{define}
Note that $V$ is a ground of $V$ by the trivial forcing.

It is known that  all grounds can be defined uniformly.

\begin{fact}[Reitz \cite{Reitz}, Fuchs-Hamkins-Reitz \cite{FHR}]\label{1}
There is a first order formula $\varphi(x,r)$ of set-theory such that:
\begin{enumerate}
\item For each $r$, the (definable) class $W_r=\{x: \varphi(x,r)\}$
is a ground of $V$.
\item For every ground $M$ of $V$,
there is $r$ such that $M=W_r$.
\end{enumerate}
\end{fact}

For this formula $\varphi$, the collection $\{W_r:r$ is a set$\}$
consists of the grounds of $V$.

\begin{define}
The {\it downward directed grounds hypothesis}, DDG,
is the assertion that, the grounds are downward directed,
that is, for every $r$ and $s$, there is $t$ such that
$W_t \subseteq W_r \cap W_s$.
The strong DDG is the assertion that,
for every set $X$,
there is $s$ such that
$W_s \subseteq \bigcap_{r \in X} W_r$.
\end{define}

Next we explicitly define how to count the grounds.

\begin{define}
We say that $V$ has \emph{only set many grounds}
if there is a set $X$ such that the collection 
$\{W_r:r \in X\}$ consists of the grounds of $V$,
that is, for every $s$, there is $r \in X$ with $W_r=W_s$.
If there is no such $X$, then $V$ has \emph{proper class many grounds}.
If $\size{X} <\ka$ ($\le \ka$, respectively), then $V$ has \emph{only $<\ka$ many grounds} (\emph{$\ka$ many grounds}, respectively).

\end{define}

\begin{define}[Fuchs-Hamkins-Reitz \cite{FHR}]
The \emph{mantle}, denoted by $\M$, is
the class $\bigcap\{W_r:r$ is a set$\}$.
\end{define}

The mantle is transitive and contains all ordinals.
We will show that the mantle must be a model of ZFC.

\begin{note}
If a poset $\bbP$ is weakly homogeneous and $G$ is $(V, \bbP)$-generic,
then the mantle of $V[G]$, $\M^{V[G]}$, is the same as the class
$\{x \in V:$\,$\Vdash_\bbP$``$x \in \M^{V[\dot G]}$''$\}$ defined in $V$
(where $\dot G$ is the canonical name for the generic filter).
Hence we can let $\M^{V^{\bbP}}$ denote
the mantle of a forcing extension of $V$ by $\bbP$.
%We will use  similar notations for other definable classes.
\end{note}

\begin{define}[Fuchs-Hamkins-Reitz \cite{FHR}]
The \emph{generic mantle}, denoted by $g\M$,
is the class 
$\bigcap \{\M^{V^{\col(\theta)}}:\theta$ is an ordinal $\}$.
\end{define}
% \begin{note}
% Since $\col(\theta)$ is weakly homogeneous and
% $\M$ is a parameter free definable class,
% the mantle of a forcing extension of $V$ by $\col(\theta)$ does not
% depend on the choice of the generic filter.

\begin{fact}[\cite{FHR}]\label{gm}
\begin{enumerate}
\item $g\M$ is the intersection of
all grounds of all forcing extensions of $V$.
\item $g\M$ is a transitive model of ZF containing all ordinals
\item $g\M$ is a forcing invariant class:
for every forcing extension $V[G]$ of $V$,
we have $g\M^V=g\M^{V[G]}$.
\end{enumerate}
\end{fact}

\section{Grounds and generic extensions}
In this section, we discuss some basic properties of grounds and generic extensions.

The following useful fact is well-known.
See Lemma 15.43 in Jech \cite{Jech} for the proof.

\begin{fact}\label{2.31}
Let $M$, $N$ be transitive models of ZFC.
If there is a poset $\bbP \in M$ and an $(M, \bbP)$-generic $G$ 
with $M \subseteq N \subseteq M[G]$,
then there is a complete suborder $\bbQ \subseteq ro(\bbP)^M$ in $M$
and an $(M, \bbQ)$-generic $H$ such that 
$M[H]=N$, and $G$ is $(N, \mathrm{ro}(\bbP)^M/H)$-generic with
$M[G]=N[G]$,
where $\mathrm{ro}(\bbP)^M$ is the completion of $\bbP$ in $M$.
In particular, $M$ is a ground of $N$, and $N$ is a ground of $M[G]$.
\end{fact}

We define the covering and the approximation properties, which
are important tools to investigate grounds and generic extensions.

\begin{define}[Hamkins \cite{Hamkins}]
Let $M \subseteq V$ be a transitive  model of ZFC containing all ordinals.
Let $\ka$ be an infinite cardinal.
\begin{enumerate}
\item 
We say that $M$ satisfies the \emph{$\ka$-covering property} for $V$
if, for every set $A$ of ordinals with $\size{A}<\ka$,
there is a set  $B \in M$ of ordinals such that $\size{B}<\ka$ and $A \subseteq B$.
\item We say that $M$ satisfies the \emph{$\ka$-approximation property} for $V$
if, for every set $A$ of ordinals, if
$A \cap x \in M$ for every set $x \in M$ of ordinals with $\size{x}<\ka$,
then $A \in M$.
\end{enumerate}
\end{define}

\begin{note}
\begin{enumerate}
% \item In the definition of the covering and the approximation properties,
% we do not require that $M$ satisfies the axiom of choice.
%\item  If $\ka$ is a cardinal and $A \in M$,
%then $\size{A}<\ka \iff \size{A}^M<\ka$.
\item $M$ satisfies the $\ka$-covering property for $V$
if and only if,
for every set $x \subseteq M$ with size $<\ka$,
there is $y \in M$ such that $x \subseteq y$ and $\size{y}<\ka$.
\item $M$ satisfies the $\ka$-approximation property for $V$
if and only if,
for every set $A \subseteq M$, if
$A \cap x \in M$ for every set $x \in M$ with $\size{x}<\ka$ then $A \in M$.
\end{enumerate}
\end{note}

% \begin{fact}\label{2.12}
% Let $M$ be a ground of $V$, and
% suppose $V=M[G]$ for some $G \subseteq \bbP \in M$.
% Let $\ka$ be a cardinal with $\size{\bbP}<\ka$.
% Then $M$ satisfies the $\ka$-covering and the $\ka$-approximation properties for $V$.
% \end{fact}

The following fact is key to the definition of the grounds as in Fact \ref{1}.
See Laver \cite{Laver2} for a proof, and
in Lemma \ref{4.3.111}, we will redo 
Laver's proof for models of ZFC$-$P.
\begin{fact}[Hamkins]\label{2}
Let $\ka$ be an uncountable  cardinal.
For transitive models $M,N\subseteq V$ of ZFC containing all ordinals,
suppose $M$ and $N$ satisfy the following properties:
\begin{enumerate}
\item $M$, $N$ satisfy the $\ka$-covering and the $\ka$-approximation properties
for $V$.
\item $\p(\ka)^M=\p(\ka)^N$.
\item $(\ka^+)^M=(\ka^+)^N=\ka^+$.
\end{enumerate}
Then $M=N$.
\end{fact}

The following is  Lemma 13 in Hamkins \cite{Hamkins}
(and see  also Mitchell \cite{Mitchell}).

\begin{fact}\label{3}
%\begin{enumerate}
% \item Let $M$ be a transitive model of ZFC containing all ordinals,
% and let $\delta$ an infinite cardinal.
% If $M$ satisfies the $\delta$-covering and the $\delta$-approximation properties
% for $V$,
% then for every cardinal $\gamma \ge \delta$,
% $M$ satisfies the $\gamma$-covering and the $\gamma$-approximation 
% properties.
%\item 
Suppose $M$ is a ground of $V$, and $M[G]=V$ for some 
$\bbP \in M$ and $(M, \bbP)$-generic $G$.
Let $\ka$ be an infinite cardinal.
If $\size{\bbP}^M<\ka$, 
then $M$ satisfies the $\ka$-covering and the $\ka$-approximation properties for $V$.
% \item Suppose $M$ is a ground for $V$, and $M[G]=V$ for some poset $\bbP \in M$ and $(M, \bbP)$-generic $G$.
% Then for every cardinal $\ka>\size{\bbP}$,
% $M$ satisfies the $\ka$-covering and the $\ka$-approximation properties for $V$.
\end{fact}

\begin{lemma}\label{3.5}
Suppose $M, N$ are grounds of $V$.
Let $\bbP \in M \cap N$ be a poset,
$G$ be an $(M, \bbP)$-generic, and $H$ be an $(N,\bbP)$-generic
such that $V=M[G]=N[H]$.
If $\la=\size{\bbP}^M=\size{\bbP}^N$ and $\p(\la)^M=\p(\la)^N$,
then $M=N$.
\end{lemma}
\begin{proof}
Let $\ka=\la^+$.
Note that $\ka=(\la^+)^M=(\la^+)^N$.
Moreover $([\ka]^{<\ka})^M=([\ka]^{<\ka})^N$
because $\p(\la)^M=\p(\la)^N$.

In order to prove $M=N$,
by Fact \ref{2}, it is sufficient to check that
the conditions (1)--(3) in Fact \ref{2} hold for $M$ and $N$.
(1) follows from Fact \ref{3.5}, and (3) is clear.

For (2), 
we prove $\p(\ka)^M \subseteq \p(\ka)^N$.
The converse direction follows from the same argument.
Fix $A \in \p(\ka)^M$. By the $\ka$-approximation property of $N$,
it is sufficient to see that $A \cap x \in M$ for every 
$x \in ([\ka]^{<\ka})^M$,
and this is immediate because
$([\ka]^{<\ka})^M=([\ka]^{<\ka})^N$.
\end{proof}

% \begin{proof}

% (2). We use Fact \ref{2}. Let $\ka=\size{\bbP}^+$.
% Note that $\ka=(\size{\bbP}^+)^M=(\size{\bbP}^+)^N$.
% By (1), $M$ and $N$ have the $\ka$-covering and the $\ka$-approximation properties.
% Since $V=M[G]=N[H]$, we have $(\ka^+)^M=(\ka^+)^N=\ka^+$.
% Thus it is enough to see that $\p(\ka)^M=\p(\ka)^N$.
% Since $\p(\bbP)^M=\p(\bbP)^N$, we have that the bounded subsets of $\ka$ in $M$ is the same as
% in $N$.
% Then we have $\p(\ka)^M=\p(\ka)^N$ by the $\ka$-approximation property of $M$ and $N$.
% \end{proof}

% \begin{fact}[Hamkins]\label{1.1}
% Let $\delta$ be a cardinal,
% and $M$ be a transitive model of ZFC containing all ordinals.
% Suppose $M$ satisfies the $\delta$-covering and the $\delta$-approximation properties.
% If there are $\alpha, \beta$ and an elementary embedding $j:M_\alpha  \to M_\beta$
% such that $\cf(\alpha)>\delta$ and the critical point of $j$ is $>\delta$,
% then $j$ is in $M$.
% \end{fact}

Next we define a strong version of the covering property.
\begin{define}
Let $\ka$ be an infinite cardinal, and $M \subseteq V$ be a
transitive model of ZFC containing all ordinals.
We say that $M$ satisfies the \emph{$\ka$-uniform covering property}\footnote{
In Friedman-Fuchino-Sakai \cite{FFS},
this property is referred as ``$M$ $\ka$-globally coves $V$''.}
 for $V$
if, for every ordinal $\alpha$ and every function $f:\alpha \to ON$,
there is a function $F \in M$ such that $\dom(F)=\alpha$, 
$f(\beta) \in F(\beta)$ and $\size{F(\beta)}<\ka$ for all $\beta <\alpha$.
\end{define}

\begin{note}
Suppose $M$ satisfies the $\ka$-uniform covering property for $V$.
\begin{enumerate}
\item For every cardinal $\la>\ka$, $M$ satisfies the $\la$-uniform covering property
for $V$.
\item If $\ka$ is regular, then $M$ satisfies the $\ka$-covering property for $V$.
\item If $\ka$ is regular,
then $M$ also satisfies a slightly stronger property:
for every ordinal $\alpha$ and every function $f:\alpha \to [M]^{<\ka}$,
there is a function $F \in M$ such that $\dom(F)=\alpha$, 
$f(\beta) \subseteq F(\beta)$, and $\size{F(\beta)}<\ka$ for all $\beta <\alpha$.
\end{enumerate}
\end{note}

The following theorem of Bukovsk\'y is very useful. It shows that 
a model of ZFC is a ground if and only if it has the uniform covering property.

\begin{fact}[Bukovsk\'y \cite{Bukovsky}]\label{Bukovsky}
Suppose $M \subseteq V$ is a transitive  model of ZFC containing all ordinals.
Let $\ka$ be a regular uncountable cardinal.
Then the following are equivalent:
\begin{enumerate}
\item $M$ satisfies the $\ka$-uniform covering property for $V$.
\item There is a poset $\bbP \in M$ and an $(M,\bbP)$-generic $G$ such that
$\bbP$ satisfies the $\ka$-c.c. in $M$, and $M[G]=V$.
\end{enumerate}
\end{fact}
See also Friedman-Fuchino-Sakai \cite{FFS} and Schindler \cite{Sch}
for simple and modern proofs of Bukovsk\'y's theorem.
Friedman-Fuchino-Sakai \cite{FFS} gave a proof using an infinitary logic,
and Schindler \cite{Sch} used an extender algebra.

%Now let us sketch the Fridman-Fuchino-Sakai's proof of Bukovsk\'y's theorem.
%If $V$ is a $\ka$-c.c. forcing extension of $V$, then it is clear that $M$ satisfies the
%$\ka$-uniform covering property for $V$.
%For the converse, we show that if $M$ satisfies the $\ka$-uniform covering property for $V$,
%then for every set $A \in V$ of ordinals, $M[A]$ is a $\ka$-c.c. forcing extension of $V$.
%Let $\gamma=\sup(A)$, and we use infinitary logic $\mathcal{L}_{\ka\gamma}$ in $M$.
%Fix $\gamma$ many propositional variables $\{C_\alpha : \alpha<\gamma\} \in M$.
%For a sentence $\varphi$, define ``$\varphi$ is true  in $A$'' as follows:
%If $\varphi$ is a propositional variable $C_\alpha$,
%then ``$C_\alpha$ is true in $A$'' $\iff \alpha \in A$.
%For other cases, we define the truth of $\varphi$ as intended.
%Let $\Phi$ be the set of all sentences in $M$.
%In $V$, define $f: ([\Phi]^{<\ka})^M \to \Phi$ as follows:
%For $\Psi \in ([\Phi]^{<\ka})^M$, if $\bigvee  \Psi$ is true in $A$,
%then $f(\Psi) \in \Psi$ such that $f(\Psi)$ is true in $A$.
%By the $\ka$-uniform covering property of $M$,
%there is $F:([\Phi]^{<\ka})^M \to ([\Phi]^{<\ka})^M$ such that $F \in M$,
%$\size{F(\Psi)}<\ka$, and $f(\Psi) \in F(\Psi)$ for every $\Psi \in \p_\ka (\Phi)^M$.
%Let $

\section{The covering and the approximation properties of models of ZFC$-$P}
In this section, we discuss the covering and the approximation 
properties of submodels of $\calH(\chi)$.
These will be used in the next section.

First let us make the following note:
\begin{note}\label{AC}
We  use the property that ``every set can be coded by a set of ordinals''.
This is implied by
ZF$-$P+ the well-ordering theorem,
and, if $\chi$ is regular uncountable, then $\calH(\chi)$ satisfies the well-ordering theorem of course.
Natural submodels of $\calH(\chi)$ satisfy only ZFC$-$P.
However, it is known that ZFC$-$P does not imply the well-ordering theorem 
(Zarach \cite{Zarach}).
So ZFC$-$P does not imply our required property.
Moreover the collection schemes do not follow from
ZFC$-$P (Zarach \cite{Zarach2}).
To avoid those difficulties, we identify the axiom of choice with the
well-ordering theorem, and the replacement schemes with the collection schemes.
Hence ``a model of ZFC$-$P'' means  ``a model of ZF$-$P+the collection schemes+the well-ordering theorem''.
\end{note}

% however, ZFC$-$P does not imply the properties which we need.
% E.g., the well-ordering theorem does not follows from ZFC$-$P. See \cite{}.
% Because of this reason, we consider only special subumodels of $\calH(\chi)$.
% \begin{define}
% Let $\chi$ be a regular cardinal.
% A submodel $M \subseteq \calH(\chi)$ 
% is \emph{almost $\calH(\chi)$}
% if there is a transitive model $W \subseteq V$  of ZFC
% such that $M=\calH(\chi)^W$.
% \end{define}
% Note that $M$ is almost $\calH(\chi)$
% is equivalent to the statement that 
% there is a set $A$ of ordinals such that $M=\calH(\chi)^{L[A]}$,
% so almost $\calH(\chi)$ is a first order expressible.

Here we explicitly define the covering and the approximation 
properties of submodels of $\calH(\chi)$.

\begin{define}
Let $\ka<\chi$ be regular uncountable cardinals.
Let $M \subseteq \calH(\chi)$ be a transitive  model of ZFC$-$P with
$\chi \subseteq M$.
\begin{enumerate}
\item $M$ satisfies the \emph{$\ka$-covering property} for $\calH(\chi)$
if for every set $A$ of ordinals with $\size{A}<\ka$ and $A \subseteq \chi$
(so $A \in \calH(\chi)$),
there is some set $B \in M$ of ordinals such that $\size{B}<\ka$
and $A \subseteq B$.
\item $M$ satisfies the \emph{$\ka$-approximation property} for $\calH(\chi)$
if for every bounded subset  $A \subseteq \chi$
(so $A \in \calH(\chi))$,
if $A \cap x \in M$ for every $x \in M \cap [\chi]^{<\ka}$,
then $A \in M$.
\item $M$ satisfies the \emph{$\ka$-uniform covering property} for $\calH(\chi)$
if for every $\alpha<\chi$ and $f:\alpha \to \chi$ (so $f \in \calH(\chi)$),
there is $F \in M$ such that $\dom(F)=\alpha$,
$f(\beta) \in F(\beta)$, and $\size{F(\beta)}<\ka$ for all $\beta<\alpha$.
\end{enumerate}
\end{define}
% \begin{note}
% Unlike the covering property of models of ZFC,
% the $\ka$-covering property of $M \subseteq \calH(\chi)$ 
% would not yields the following stronger property:
% for every set $x \in [M]^{<\ka}$ (so $x \in \calH(\chi))$,
% there is $y \in M$ such that $x \subseteq y$ and  $\size{y}<\ka$.
% \end{note}

The following is a variation of Fact \ref{2}.

\begin{lemma}\label{4.3.111}
Let $\ka<\chi$ be regular uncountable cardinals with $\ka^+<\chi$.
Let $M, N \subseteq \calH(\chi)$ be 
transitive models of ZFC$-$P with
$\chi \subseteq M \cap N$.
If $M$ and $N$ satisfy the $\ka$-uniform covering and the $\ka$-approximation
properties for $\calH(\chi)$,
and $\p(\ka) \cap M=\p(\ka) \cap N$,
then $M=N$.
\end{lemma}
\begin{proof}
We  repeat the proof of Fact \ref{2}.

First note that $(\ka^+)^M=(\ka^+)^N=\ka^+$;
this is immediate from the $\ka$-uniform covering property of $M$ and $N$.
Since $\p(\ka) \cap M=\p(\ka) \cap N$,
we have that $\p(\alpha ) \cap M=\p(\alpha) \cap N$ for every $\alpha<\ka^+$, which
we can verify as follows:
For $\alpha<\ka^+$, we can take a set $X \subseteq \ka \times \ka$
such that $X \in M$ and $\seq{\ka, X}$ is isomorphic to $\alpha$.
There is a unique isomorphism $\pi$ from $\seq{\ka, X}$ onto $\alpha$.
Since $\p(\ka) \cap M =\p(\ka) \cap N$,
we have $X \in N$, hence $\pi$ is also in $N$.
Using this $\pi$, we can check that $\p(\alpha) \cap M=\p(\alpha) \cap N$.
We also note that $M$ and $N$ satisfy the $\ka$-covering property for $\calH(\chi)$.

Since both $M$ and $N$ are models of ZFC$-$P,
every set in $M$ and $N$ can be coded by a set of ordinals
(see Note \ref{AC}).
Thus, in order to show $M=N$,
it is sufficient to show that,
for every set $A \subseteq \chi$,
we have $A \in M \iff A \in N$.

First we show that $[\chi]^{<\ka} \cap M=[\chi]^{<\ka} \cap N$.
Fix $A \in [\chi]^{<\ka} \cap M$. We see $A \in N$,
and the converse follows from the same argument.
By induction on $i<\ka$,
we take $\seq{x_i,y_i:i<\ka}$ as follows:
\begin{enumerate}
\item $x_i \in [\chi]^{<\ka} \cap M$ and $y_i \in [\chi]^{<\ka} \cap N$.
\item $\seq{x_i:i<\ka}$ and $\seq{y_i:i<\ka}$ are $\subseteq$-increasing.
\item $A \subseteq  x_i \cup y_i \subseteq x_{i+1} \cap y_{i+1}$.
\end{enumerate}
Suppose $\seq{x_j,y_j:j<i}$ is defined.
Let $z=\bigcup_{j<i} (x_j \cup y_j)$.
We have $z \in [\chi]^{<\ka}$.
By the $\ka$-covering property of $M$ and $N$,
we can find $x_i \in M \cap [\chi]^{<\ka}$ and $y_i \in N \cap [\chi]^{<\ka}$
with $z \subseteq x_i, y_i$.

Now let $B=\bigcup_{i<\ka} x_i=\bigcup_{i<\ka} y_i$.
Next we check $B \in M \cap N$.
To see $B \in M$, we use the $\ka$-approximation property of $M$.
Thus take $c \in M \cap [\chi]^{<\ka}$.
Since $\size{c}<\ka$ and $B=\bigcup_{i<\ka} x_i$, there must be $k<\ka$ with
$B \cap c \subseteq x_k$.
Then $B \cap c=x_k \cap x \in M$. Hence $B \in M$ by the approximation property of $M$.
We also know that $B \in N$ by the same argument.

Let $\delta=\ot(B)<\ka^+$.
Let $\pi:B \to \delta$ be the transitive collapse map.
We know $\pi \in M \cap N$.
Let $D=\pi``A \in M \cap \p(\delta)$.
Since $\p(\delta) \cap M=\p(\delta) \cap N$,
we have that $D \in N$.
Then $A=\pi^{-1}``D \in N$.

Now we have $[\chi]^{<\ka} \cap M=[\chi]^{<\ka} \cap N$.
Finally we show that for every set $A \subseteq \chi$,
we have $A \in M \iff A \in N$.
Take a set $A \in M \cap \p(\chi)$.
For every $x \in [\chi]^{<\ka} \cap M$,
we have that $A \cap x$ is in $M$.
On the other hand,
since
$[\chi]^{<\ka} \cap M=[\chi]^{<\ka} \cap N$,
we have that $A \cap x \in N$ for every $x \in [\chi]^{<\ka}$.
Then $A \in N$ by the $\ka$-approximation property of $N$.
The converse  follows from the same argument.
\end{proof}

By Facts \ref{3} and \ref{Bukovsky},
if $M$ and $\calH(\chi)$ are models of ZFC,
 and $M$ is a class of $\calH(\chi)$,
then the uniform covering property of $M$ implies the approximation property
using Bukovsk\'y's theorem.
However we do not know that Fact \ref{Bukovsky} is valid for
models of ZFC$-$P, % or even $\calH(\chi)$,
so the implication is not clear for submodels of $\calH(\chi)$.
Here, we give a direct proof of 
the approximation property of submodels of  $\calH(\chi)$ from
the uniform covering property.

To  do it, we make some observations about trees.
For a tree $T$ of height $\alpha$ and $\beta<\alpha$,
let $T_\beta$ be the $\beta$-th level of $T$.

The following is a well-known theorem of Kurepa:
\begin{fact}[Kurepa \cite{Kurepa}]\label{Kurepa}
Let $\ka<\la$ be regular cardinals.
For every tree $T$ of height $\la$,
if $\size{T_\alpha}<\ka$ for every $\alpha<\la$,
then $T$ has a cofinal branch.
\end{fact}

\begin{lemma}\label{2.4.111}
Let $\ka<\la$ be regular cardinals.
Let $T$ be a tree of height $\la$.
Suppose that 
$\size{T_\alpha}<\ka$
for every $\alpha<\la$.
Then $T$ has fewer than $\ka$ many cofinal branches.
\end{lemma}
\begin{proof}
Suppose to the contrary that
$T$ has $\ka$ cofinal branches $\seq{B_i:i<\ka}$.
For each $i<j<\ka$,
there is some $\alpha(i,j)<\la$
such that $B_i \cap T_{\alpha(i,j)} \neq B_j \cap T_{\alpha(i,j)}$.
Put $\alpha=\sup\{\alpha(i,j):i<j<\ka\}$.
We have $\alpha<\la$ since $\ka<\la$.
Then, for every $i<j<\ka$,
we know $B_i \cap T_\alpha \neq B_j \cap T_\alpha$.
Thus $T_\alpha$ has cardinality at least $\ka$,
which is a contradiction.
\end{proof}

\begin{lemma}\label{2.5.111}
Let $\ka$ be a regular cardinal, and
$\mu>\ka$ an ordinal with $\cf(\mu)>\ka$.
Let $T$ be a tree of height $\mu$.
Suppose that $\size{T_\alpha}<\ka$ for every $\alpha<\mu$.
Let $W \subseteq V$ be a transitive model of ZFC containing all ordinals.
If $T \in W$,
then every cofinal branch of $T$ in $V$ belongs to $W$.
% if $\cf(\mu)^V$ is regular in $W$,
% then every cofinal branch of $T$ in $W$ belongs to $V$.
\end{lemma}
\begin{proof}
Case 1: $\cf(\mu)^W=\mu$.
We work in $W$. By Lemma \ref{2.4.111},
we have that
$T$ has fewer than $\ka$ many cofinal branches.
Let $\nu<\ka$ and $\seq{B_i:i<\nu}$ be an enumeration of all
cofinal branches of $T$.

In $V$, suppose to the contrary that
there exists a cofinal branch $B$ of $T$ with $B \notin W$.
For each $i<\nu$, there is $\alpha(i)<\mu$
such that $T_{\alpha(i)} \cap B \neq T_{\alpha(i)} \cap B_i$.
Since $\cf(\mu)^V>\ka$,
we have that $\alpha=\sup\{\alpha(i)+1:i<\nu\}<\mu$.
Pick $t \in B \cap T_{\alpha}$.
We know $t \notin B_i$ for every $i<\nu$.
On the other hand,
by Kurepa's theorem (Fact \ref{Kurepa}) applied to $W$,
there is a cofinal branch $B' \in W$ with $t \in B'$.
Then $B'=B_i$ for some $i<\nu$,  but $t \in B_i$,
which is a contradiction.

Case 2: $\cf(\mu)^W<\mu$.
We work in $W$.
Fix a cofinal set $X \subseteq \mu$ with order type $\cf(\mu)^W$,
and consider the subtree $T^*=\bigcup_{\alpha \in X} T_\alpha$ of $T$.
Then $T^*$ is a tree of height $\cf(\mu)^W$,
and $\size{T^*_\alpha}<\ka$ for every $\alpha<\cf(\mu)$.
Moreover, if $B^* \subseteq T^*$ is a cofinal branch of $T^*$,
then there is a unique cofinal branch $B \subseteq T$ with
$B \cap T^*=B^*$.
Now let $B \subseteq T$ be a cofinal branch of $T$ with $B \in V$.
Let $B^*=B \cap T^*$. $B^*$ is a cofinal branch of $T^*$.
By (1), we have that $B^* \in W$,
and there is a cofinal branch $B' \subseteq T$ with $B' \cap T^*=B^*$.
Then it is clear that $B'=B$, hence $B \in W$.
\end{proof}

Now we prove the main result of this section.
\begin{lemma}\label{4.4.111}
Let $\theta$ be a strong limit cardinal,
and $\ka<\theta$ be a regular uncountable cardinal.
Let $\chi=\theta^+$, and suppose
there is a transitive model $W \subseteq V$ of ZFC
such that $M=\calH(\chi)^W=\calH(\chi) \cap W$.
If $M$ satisfies the $\ka$-uniform covering property for $\calH(\chi)$,
then $M$ satisfies the $\ka^+$-approximation property for $\calH(\chi)$.
\end{lemma}
\begin{proof}
First note the following:
\begin{enumerate}
\item For every $\alpha$ with $\ka \le \alpha<\chi$,
$\alpha$ is regular in $M$ if and only if $\alpha$ is regular in $V$.
\item In particular, if $\ka \le \alpha<\chi$,
then $\alpha$ is a cardinal in $M$ if and only if $\alpha$ is a cardinal in $V$.
\item The set $M \cap {}^{<\theta} 2$ is in $M$.
\end{enumerate}
(1) and (2) follow from the $\ka$-uniform covering property of $M$.
For (3),
since $\theta$ is strong limit,
$\theta$ is also a strong limit cardinal in $W \subseteq V$.
Thus $W \cap {}^{<\theta} 2$ $(=({}^{<\theta} 2)^W)$
has cardinality $\theta$. % in $W$.
%$M$ is of the form $\calH(\chi)$ in $W$,
Hence we have 
\[
M \cap {}^{<\theta} 2 
=W \cap {}^{<\theta} 2
\in \calH(\chi) \cap W=M.
\]

We prove the following by induction on $\alpha<\chi$:
For every $A \subseteq \alpha$,
if $A \cap x \in M$ for every $x \in M \cap [\alpha]^{<\ka^+}$,
then $A \in M$.

Fix $\alpha<\chi$, $A \subseteq \alpha$,
and suppose $A \cap x \in M$
for every $x \in M \cap [\alpha]^{<\ka^+}$.
By the induction hypothesis,
we have that $A \cap \beta \in M$ for every $\beta <\alpha$.

It is clear in the case $\alpha<\ka^+$.
Thus suppose $\alpha \ge \ka^+$.
\\

Case 1: $\alpha$ is not a cardinal.
By the above remark (2), we have that $\alpha$ is not a cardinal in $M$.
Put $\la=\size{\alpha}^M$.
Take a bijection $\pi:\la \to \alpha$ with $\pi \in M$,
and let $B=\pi^{-1}``A \subseteq \la$.
It is easy to check that $B \cap x \in M$
for every $x \in [\la]^{<\ka} \cap M$.
By the induction hypothesis,
we have $B  \in M$. Then $A=\pi``B \in M$.
\\

Case 2: $\alpha$ is a cardinal with $\cf(\alpha)<\ka^+$.
Note that $\alpha \le \theta$ by remark (2).

First we claim the following:
\begin{claim}
The set $M \cap [\alpha]^{<\ka^+}$ is stationary in
$[\alpha]^{<\ka^+}$
\end{claim}
\begin{proof}
First note that  we do not require that $M \cap [\alpha]^{<\ka^+} \in M$.

To show the assertion, take a function  $f:[\alpha]^{<\om} \to \alpha$.
We will find $x \in M \cap [\alpha]^{<\ka^+}$ which is closed under $f$ and $\ka \subseteq x$.
There is a bijection $\pi \in M$ from $[\alpha]^{<\om}$ onto $\alpha$.
Hence, by the $\ka$-uniform covering property of $M$,
there is $F \in M$ such that $\dom(F)=[\alpha]^{<\om}$,
$f(s) \in F(s) \subseteq \alpha$, and $\size{F(s)}<\ka$.
Since $M$ is a model of ZFC$-$P,
we can find $x \in M \cap [\alpha]^{<\ka^+}$ such that $\ka \subseteq x$ and
$F(s) \subseteq x$ for all $s \in [x]^{<\om}$.
Then $x$ is clearly closed under $f$.
\qedhere[Claim]
\end{proof}

By remark (3),
we have that $M \cap \{x \subseteq \alpha:x$ is bounded in $\alpha\}  \in M$.
Fix $\mu \in \chi$ and 
a one-to one enumeration  $\seq{B_i:i<\mu} \in M$ of
$M \cap \{x \subseteq \alpha:x$ is bounded in $\alpha\}$.
Fix $X \subseteq \alpha $ such that $X \in M$,
$\size{X}<\ka^+$, and $X$ is unbounded in $\alpha$.
We have that $A \cap \beta \in M$ for every $\beta<\alpha$.
Thus, for each $\beta \in X$,
there is a unique $i(\beta)<\mu$ with
$A \cap \beta=B_{i(\beta)}$.
We will see that the set $\{i(\beta) :\beta \in X\}$ is in $M$,
then $A=\bigcup \{B_{i(\beta)}:\beta \in X\} \in M$, as required.

Since $M$ satisfies the $\ka$-uniform covering property,
$M$ satisfies the $\ka^+$-covering property.
Thus we can find $Y \in M \cap [\mu]^{<\ka^+}$ with $\{i(\beta):\beta \in X\} \subseteq Y$.

By the claim above, we can find $N \prec \calH(\chi)$
such that $\size{N}<\ka^+$, $\ka \subseteq N$,
$\sup(N \cap \alpha)=\alpha$, $\seq{B_i:i<\mu}, X, Y, A \in N$, 
and $N \cap \alpha \in M$.
Note that $X, Y \subseteq N$ since $\size{X},\size{Y}<\ka^+$.

For every $\beta \in X$ and
$i \in Y$, we have that 
$i = i(\beta) \iff
A \cap \beta =B_i$.
%We know $A \cap \beta \cap  N =B_{i(\beta)} \cap (N \cap \alpha)$ for $\beta \in X$.
By the elementarity of $N$ and $X, Y \subseteq N$,
for every $\beta \in X$ and
$i \in Y$, we know
that 
$i = i(\beta) \iff
A \cap \beta \cap N=B_i \cap (N \cap \alpha)$.

Since $\size{N}<\ka^+$ and $N \cap \alpha \in M$,
we have $N \cap \alpha \in M \cap [\alpha]^{<\ka^+}$,
hence $A \cap N =A \cap (N \cap \alpha) \in M$.
Thus, in $M$, for each $\beta \in X$, $i(\beta)$ is definable
as the unique $i \in Y$ with
$(A  \cap N) \cap \beta=B_i \cap (N \cap \alpha)$.
Hence we have $\seq{i(\beta):\beta \in X} \in M$,
and $A=\bigcup_{\beta \in X} B_{i(\beta)} \in M$.
\\

Case 3: $\alpha$ is a cardinal with $\cf(\alpha) \ge \ka^+$.
As in Case 2, we have that $\alpha \le \theta$.

Let $f:\alpha \to 2$ be the characteristic function of $A$.
It is enough to see that $f \in M$.
Since $A \cap \beta \in M$ for every $\beta <\alpha$,
we have that $f \restriction \beta\in M$ for every $\beta<\alpha$.
Since $\alpha \le \theta$,
the set ${}^{<\alpha} 2 \cap M$ is in $M$ by remark (3).
Fix $\mu<\chi$ and a one-to-one enumeration $\seq{g_i:i<\mu} \in M$
of ${}^{<\alpha} 2 \cap M$.
Define $h:\alpha \to \mu$ by
$h(\beta)=i \iff f \restriction \beta=g_i$.
By the $\ka$-uniform covering property of $M$,
there is $H \in M$ with $h(\beta) \in H(\beta)$ and
$\size{H(\beta)}<\ka$. We may assume that
for every $i \in H(\beta)$, $\dom(g_i)=\beta$.
Put $T'=\{g_i: i \in H(\beta), \beta <\alpha\} \in M$.
We know that $f \restriction \beta \in T'$ for all $\beta<\alpha$.

In $M$, let $T$ be the set of all $g \in T'$
such that $g \restriction \gamma \in T'$ for every $\gamma<\dom(g)$.
Again, we have $f \restriction \beta \in T$ for every $\beta<\alpha$.
Thus $T$ is a tree of height $\alpha$ and 
 for every $g \in T$ and $\gamma<\dom(g)$, we have $g \restriction \gamma \in T$.
The $\beta$-th level of $T$ is a subset of $\{g_i: i \in H(\beta)\}$.
Since $\size{H(\beta)}<\ka$,
we have that $\size{T_\beta}<\ka$ for every $\beta<\alpha$.
The set $B=\{f \restriction \beta:\beta<\alpha\}$ is a cofinal branch of $T$.
$T$ is a tree in a transitive model $W$ of ZFC.
Applying Lemma \ref{2.5.111} to $W$ and $T$, we have $B \in W$.
Then $B \in M$ since $M=\calH(\chi)^W$, hence $f \in M$.
\end{proof}

We have the following as an immediate corollary of Lemma \ref{4.4.111}.
%We do not use this corollary later.
\begin{cor}\label{approx}
Let $W \subseteq V$ be a transitive model of ZFC
containing all ordinals.
Let $\ka$ be a regular uncountable cardinal.
If $W$ satisfies the $\ka$-uniform covering property for $V$,
then $W$ satisfies the $\ka^+$-approximation property for $V$,
and $W$ is definable in $V$ with some parameters.
\end{cor}
\begin{proof}
For each strong limit cardinal $\theta>\ka$,
we have that $\calH(\theta^+)^W$ is of the form $W \cap \calH(\theta^+)$ and satisfies the $\ka^+$-approximation property
for $\calH(\theta^+)$ by Lemma \ref{4.4.111}.
Then it is clear that $W$ satisfies the $\ka^+$-approximation property for $V$.

Let $r=\p(\ka^{+}) \cap W$.
Now $W$ is definable in $V$ with parameters $r$ and $\ka^+$ as follows:
$x\in W$ if and only if there is a strong limit cardinal $\theta>\ka$
and the unique transitive model $M \subseteq \calH(\theta^+)$ of ZFC$-$P
such that $\theta^+ \subseteq M$,
$M \cap \p(\ka^+)=r$, $M$ satisfies the $\ka^+$-uniform covering and the
$\ka^+$-approximation properties for $\calH(\theta^+)$,
and $x \in M$.
\end{proof}

\begin{note}
\begin{enumerate}
\item If $W \subseteq V$ satisfies the uniform covering property for $V$,
then $W$ is a ground by Bukovsky's theorem,
hence $W$ is definable in $V$ by Laver and Woodin's result.
In this sense, the definablity of $W$  satisfying the uniform covering property is not a new result.
\item On the other hand,
in the proof of Corollary \ref{approx},
it is not necessary that
$V$ satisfies the replacement scheme for the formulas
of the language $\{\in, W\}$.
So Corollary \ref{approx} shows that
every $W \subseteq V$ satisfying 
the uniform covering property
is definable,
regardless of
the replacement scheme for the formulas
of the language $\{\in, W\}$.
\item The $\ka^+$-approximation property in the previous corollary cannot be strengthened to the 
$\ka$-approximation property;
suppose there exists a $\ka$-Suslin tree $T$,
and let $\bbP$ be a $\ka$-c.c. poset adding a cofinal branch of $T$.
If $G$ is $(V, \bbP)$-generic, then 
$V$ satisfies the $\ka$-uniform covering property for $V[G]$,
but a cofinal branch of $T$ witnesses that $V$ does not satisfy the $\ka$-approximation property for $V[G]$.
\item On the other hand, 
if $W\subseteq V$ satisfies the $\ka$-uniform covering property for $V$ 
(hence $V$ is a $\ka$-c.c. forcing extension of $W$) and
there is no $\ka$-Suslin tree in $W$,
then $W$ satisfies the $\ka$-approximation property for $V$
(see Usuba \cite{Usuba}).
\end{enumerate}
\end{note}

\section{The strong DDG}
In this section,
we prove that 
the strong DDG is a theorem of ZFC.

\begin{prop}\label{main 11}
The strong DDG holds.
\end{prop}
\begin{proof}
Fix a set $X$. We will construct a ground $W$ of $V$ such that
$W \subseteq W_r$ for all $r \in X$ simultaneously as follows.
Then $W$ is a ground of each $W_r$ by Fact \ref{2.31}.

For each $r \in X$,
there is a poset $\bbP_r \in W_r$
and a $(W_r, \bbP_r)$-generic $G_r$ with $V=W_r[G_r]$.
Fix a regular uncountable $\ka$ such that
$\size{X}<\ka$ and 
$\size{\bbP_r}<\ka$ for every $r \in X$.
Then by Facts \ref{3} and \ref{Bukovsky}, 
each $W_r$
satisfies the $\ka$-uniform covering and the $\ka$-approximation
properties for $V$.

Fix a strong limit cardinal $\theta>\ka$,
and 
let $\chi=\theta^+$.
First we show that there exists a transitive model $M$ of ZFC$-$P
such that $\chi \subseteq M \subseteq \calH(\chi)$
and $M$ satisfies the $\ka^{++}$-uniform covering and the $\ka^{++}$-approximation properties
for $\calH(\chi)$.

Let $\gamma=\chi^{<\chi}$,
and $\seq{f_\xi:\xi<\gamma}$ be an enumeration of ${}^{<\chi} \chi$.
% For each $f_\xi$ and each $W_r$,
% we can find a covering function $F \in W_r$,
% that is, $\dom(f_\xi)=\dom(F)$, $f(\beta) \in F(\beta)$, and
% $\size{F(\beta)}<\ka$.
Now define $h:\chi \times \gamma  \to \chi$ as follows:
$h(\alpha,\xi)=f_\xi(\alpha)$ if $\alpha \in \dom(f_\xi)$,
and $h(\alpha,\xi)=0$ otherwise.
\begin{claim}
There is a function $H$ such that $\dom(H)=\chi \times \gamma$,
$h(\alpha,\xi) \in H(\alpha,\xi) \subseteq \chi$,
$\size{H(\alpha,\xi)} < \ka^+$ for $\alpha<\chi, \xi<\gamma$,
and $H \in \bigcap_{r \in X} W_r$.
\end{claim}
\begin{proof}
By induction on $i<\ka$,
we define $H_{i,r}$ for $r \in X$ as follows:
\begin{enumerate}
\item $H_{i,r}$ is a function with $H_{i,r} \in W_r$.
\item $\dom(H_{i,r}) =\chi \times \gamma$,
$h(\alpha,\xi) \in H_{i,r}(\alpha,\xi) \subseteq \chi$, and
$\size{H_{i,r}(\alpha,\xi)}<\ka$
for all $\alpha <\chi$ and $\xi<\gamma$.
\item For all $\alpha<\chi$ and $\xi<\gamma$,
$\bigcup_{j<i, s \in X, }H_{j,s}(\alpha,\xi) \subseteq H_{i,r}(\alpha,\xi)$.
\end{enumerate}
Let $i<\ka$ and suppose $H_{j,s}$ was defined for all $s \in X$ and $j<i$.
Fix $r \in X$. To define $H_{i,r}$.
let $H'$ be such that $H'(\alpha,\xi)=\bigcup_{j<i, s \in X}H_{j,s}(\alpha,\xi)$
for all $\alpha <\chi$ and $\xi<\gamma$.
Since $\size{X}$ and $i$ are less than $\ka$,
we have that $\size{H'(\alpha,\xi)}<\ka$
for all $\alpha <\chi$ and $\xi<\gamma$.
$W_r$ satisfies the $\ka$-uniform covering property for $V$,
hence we can find $H_{i,r} \in W_r$ such that
$H'(\alpha,\xi) \subseteq H_{i,r}(\alpha,\xi) \subseteq \chi$
and $\size{H_{i,r}(\alpha,\xi)}<\ka$
for all $\alpha <\chi$ and $\xi<\gamma$.
It is clear that $h(\alpha,\xi) \in H_{i,r}(\alpha,\xi)$.

Finally, let $H$ be such that $H(\alpha,\xi)=\bigcup_{i<\ka}H_{i,r}(\alpha,\xi)$
for some (in fact all) $r \in X$.
Clearly $h(\alpha,\xi) \in H(\alpha,\xi)$ and $\size{H(\alpha,\xi)}<\ka^+$.
We have to show that $H \in W_r$ for all $r \in X$.

Fix $r \in X$. 
Let $E=\{\seq{\alpha,\xi, \eta} \in \chi\times \gamma\times \chi: \eta \in H(\alpha,\xi)\}$.
It is sufficient to show that $E \in W_r$.
To show that $E \in W_r$, we use the $\ka$-approximation property
of $W_r$.
So take $a \in [\chi \times \gamma \times \chi]^{<\ka} \cap W_r$. We will  see that $a \cap E \in W_r$.
Let $d=\{\seq{\alpha,\xi} \in \chi \times \gamma: \exists \eta\,(\seq{\alpha,\xi,\eta} \in
a \cap E)\}$.
We have $\size{d}<\ka$.
For each $\seq{\alpha, \xi} \in d$,
the set $\{\eta<\chi: \seq{\alpha,\xi, \eta} \in a \cap E\}
\subseteq H(\alpha,\xi)$
has  cardinality $<\ka$,
thus we can find $i(\alpha,\xi)<\ka$
with $\{\eta<\chi:\seq{\alpha,\xi,\eta} \in a \cap E\}
\subseteq H_{i(\alpha,\xi),r}(\alpha,\xi)$.
% Since $W_r$ satisfies the uniform $\ka$-covering property for $V$,
% $W_r$ also satisfies the $\ka$-covering property for $V$.
% Thus we can find $e \in [\la^2]^{<\ka} \cap W_r$ with $d \subseteq e$.
% Since $\size{a}, \size{e}<\ka$,
Put $i^*=\sup\{i(\alpha,\xi):\seq{\alpha,\xi} \in d\}<\ka$.
We have $\{\eta<\chi: \seq{\alpha,\xi,\eta} \in a \cap E\} \subseteq H_{i^*,r}(\alpha,\xi)$
for all $\seq{\alpha, \xi} \in d$.
Then $a \cap E=\{\seq{\alpha,\xi,\eta} \in a: \eta \in H_{i^*,r}(\alpha,\xi) \} \in W_r$.
\qedhere[Claim]
\end{proof}

Fix a bijection $\pi:\chi \times \gamma \times \chi \to \gamma$ with $\pi \in L$,
and let $A=\pi``H$.
Clearly $A \in W_r$ for all $r \in X$.
Since $A$ is a set of ordinals, we have that $L[A]$ is a model of ZFC,
and $H \in L[A]$.
Now let $M=\calH(\chi)^{L[A]}$.
We know $M=L[A] \cap \calH(\chi)$ and $\chi \subseteq M \subseteq \calH(\chi)$,
and since $A \in W_r$ for all $r \in X$, we have $M \subseteq \bigcap_{r \in X} W_r$.

\begin{claim}
$M$ satisfies the $\ka^+$-uniform covering property for $\calH(\chi)$.
\end{claim}
\begin{proof}
Take $f:\alpha \to \chi$ for some $\alpha<\chi$.
Since $\seq{f_\xi:\xi<\gamma}$ is an enumeration of ${}^{<\chi} \chi$,
there is some $\xi^*<\gamma$ with $f=f_{\xi^*}$.
Now define the function $F$ such that $\dom(F)=\alpha$
and $F(\beta)=H(\beta,\xi^*)$.
By the choice of $H$, we have that $f(\beta) =f_{\xi^*}(\beta) \in H(\beta,\xi^*)=F(\beta)$ 
and $\size{F(\beta)}<\ka^+$.
Moreover, since $H \in L[A]$, we have $F \in L[A]$,
and $F \in \calH(\chi)^{L[A]}=M$.
\qedhere[Claim]
\end{proof}
$M$ is $\calH(\chi) \cap L[A]$ and $\chi$ is 
the successor of the strong limit cardinal $\theta$.
So $M$ satisfies the $\ka^{++}$-approximation property for $\calH(\chi)$
by Lemma \ref{4.4.111}.

In summary, for each strong limit cardinal $\theta>\ka$,
we can find a transitive model $M$ of ZFC$-$P such that
$\theta^+ \subseteq M \subseteq \calH(\theta^+)$, $M \subseteq \bigcap_{r \in } W_r$,
and $M$ satisfies the $\ka^{++}$-uniform covering and the $\ka^{++}$-approximation properties 
for $\calH(\theta^+)$. Note that, by Lemma \ref{4.3.111} (and Note \ref{AC}), letting $r =\p(\ka^{++}) \cap M$,
$M$ is a unique transitive model $N$ of ZFC$-$P
such that $\theta^+ \subseteq N \subseteq \calH(\theta^+)$,
$r=\p(\ka^{++}) \cap N$, and
$N$ satisfies the $\ka^{++}$-uniform covering and the $\ka^{++}$-approximation properties 
for $\calH(\theta^+)$.

For each $p \subseteq \p(\ka^{++})$,
let $I_p$ be the class of all strong limit cardinals $\theta>\ka$
such that
there exists
a transitive model $M$ of ZFC$-$P such that
$\theta^+ \subseteq M \subseteq \calH(\theta^+)$, 
$p=M \cap \p(\ka^{++})$,
$M \subseteq \bigcap_{r \in } W_r$,
and $M$ satisfies the $\ka^{++}$-uniform covering and the $\ka^{++}$-approximation properties 
for $\calH(\theta^+)$.
By the pigeonhole argument,
there must be one set $p^* \subseteq \p(\ka^{++})$ such that
$I_{p^*}$ forms a proper class.
For each $\theta \in I_{p^*}$,
fix a unique transitive model $M^{\theta}$ of ZFC$-$P
such that $\theta^+ \subseteq M^\theta \subseteq \calH(\theta^+)$,
$M^\theta \subseteq  \bigcap_{r \in X} W_r$,
$p^*=\p(\ka^{++}) \cap M^\theta$, and
$M^\theta$ satisfies the $\ka^{++}$-uniform covering and the $\ka^{++}$-approximation properties 
for $\calH(\theta^+)$.
The sequence $\seq{ M^{\theta}:\theta  \in I_{p^*}}$ is coherent:
For $\theta<\theta'$ from $I_{p^*}$,
we have $M^\theta=M^{\theta'} \cap \calH(\theta^+)$.
This can be verified as follows:
Since $M^{\theta'} \cap \calH(\theta^+)=\{x \in M^{\theta'}:
M^{\theta'} \vDash$``$\size{\mathrm{trcl}(x)}<\theta^+$''$\}
\in M^{\theta'}$,
it is easy to check that
$M^{\theta'} \cap \calH(\theta^+)$
is a transitive model of ZFC$-$P and satisfies the
$\ka^{++}$-uniform covering property
for $\calH(\theta^+)$.
Then $M^{\theta}=M^{\theta'} \cap \calH(\theta^+)$
by Lemma \ref{4.3.111}.

% Moreover, for $\theta, \theta' \in I_{p^*}$ with $\theta<\theta'$,
% since $M^{\theta'} \cap \calH(\theta^+)=\{x \in M^{\theta'}:
% M^{\theta'} \vDash$``$\size{\mathrm{trcl}(x)}<\theta^+$''$\} \in
%  M^{\theta'}$, 
% we have $M^{\theta} \in M^{\theta'}$.

Finally let $W=\bigcup\{M^\theta:\theta \in I_{p^*}\}$.
Note that $W$ is definable in $V$.
\begin{claim}
$W$ is a transitive model of ZFC containing all ordinals.
\end{claim}
\begin{proof}
Transitivity is clear, and it is also clear that $W$ contains all ordinals.
To show that $W$ is a model of ZFC,
by Fact \ref{model},
it is enough to see that 
$N$ is closed under the G\"odel operations, $N$ is almost universal,
and $N$ satisfies the axiom of choice.

For every $\theta \in I_{p^*}$,
$M^\theta$ is closed under the G\"odel operations.
Since $W=\bigcup\{M^\theta:\theta \in I_{p^*}\}$ and $\seq{M^\theta:\theta \in I_{p^*}}$ is coherent,
we have that $W$ is closed under the G\"odel operations.
% Moreover, for $\theta, \theta'$ from $I_{p^*}$ with $\theta<\theta'$,
% we know $M^\theta \subseteq M^{\theta'} \subseteq W$.
The almost universality of $W$ is immediate from the coherency of
$\seq{M^\theta:\theta \in I_{p^*}}$.

We know that $W$ is a model of ZF. 
For the axiom of choice in $W$,  for each $x \in W$,
there is $\theta$ with $x \in M^\theta$.
$M^\theta$ is a model of ZFC$-$P,
thus $M^\theta$ has a well-ordering on $x$. This well-ordering belongs to $W$,
hence $M$ satisfies the axiom of choice as well.
\qedhere[Claim]
\end{proof}

We know that  $W$ is a transitive model of ZFC containing all ordinals.
Clearly $W \subseteq \bigcap_{r \in X} W_r$.
Moreover, $W$ satisfies the $\ka^{++}$-uniform covering property for $V$.
Then $W$ is a  ground of $V$ by Fact \ref{Bukovsky}. This completes the proof.
\end{proof}

We know that the strong DDG is a theorem of ZFC.
So the following fundamental properties become theorems of ZFC as well.
Recall that a \emph{bedrock} is a minimal ground, and a \emph{solid bedrock} is
a minimum ground (\cite{Reitz}, \cite{FHR}).

\begin{cor}\label{4.111}
\begin{enumerate}
\item $\M$ is a transitive model of ZFC and $g\M=\M$.
\item $\M$ is a forcing invariant class, that is,
for every forcing extension $V[G]$ of $V$,
we have $\M=\M^{V[G]}$.
\item The following are equivalent:
\begin{enumerate}
\item Every forcing extension of $V$ has only set many grounds.
\item $V$ has only set many  grounds.
\item $\M$ is a solid bedrock of all forcing extensions of $V$.
\item $\M$ is a solid bedrock of $V$.
\item $\M$ is a ground of $V$.
\item $V$ has a bedrock.
\end{enumerate}
\end{enumerate}
\end{cor}
\begin{proof}
(1). In Fuchs-Hamkins-Reitz \cite{FHR},
it was shown that if the strong DDG holds then
$\M$ is a model of ZFC (Theorem 22 in \cite{FHR}).
Now we proved the strong DDG is a theorem of ZFC, so $\M$ is a model of ZFC.
By Corollary 51 in \cite{FHR}, we know that, if the DDG holds in all forcing extensions of $V$,
then $g\M$ is the same as $\M$.
Again, since the strong DDG is a theorem of ZFC,
the DDG holds in all forcing extensions of $V$.
Hence we have that $g\M=\M$.

(2) follows from (1) and Fact \ref{gm}.

(3). (a) $\Rightarrow$ (b) is trivial.

(b) $\Rightarrow$ (c). 
If $V$ has only set many grounds, then by the strong DDG,
$\M$ is a ground of $V$, hence is a solid bedrock of $V$.
By (1), for every forcing extension $V[G]$,
we have $\M^V=g\M \subseteq V \subseteq V[G]$, hence $\M^V=g\M$ is a ground of $V[G]$,
and, by the definition of $g\M$, $\M^V$ is a solid bedrock of $V[G]$.

(c) $\Rightarrow$ (d) $\Rightarrow$ (e) $\Rightarrow$ (f) are also trivial.

(f) $\Rightarrow$ (e).  Suppose $W \subseteq V$ is a bedrock.
If $\M \subsetneq W$, then there is a ground $W'$ and $x \in W$
with $x \notin W'$. By the DDG,
$W$ and $W'$ have a common ground $\overline{W}$.
We know $W=\overline{W}$ by the minimality of $W$,
so $x \notin W$. This is a contradiction.
Hence we have $\M=W$ is a ground of $V$.

(e) $\Rightarrow$ (c).
We see that $\M$ is a solid bedrock of all forcing extensions $V[G]$ of $V$.
Since $\M^V$ is a ground of $V$ and $V$ is a ground of $V[G]$,
we have that $\M^V$ is a ground of $V[G]$.
If $W \subseteq V[G]$ is a ground of $V[G]$,
then by the DDG in $V[G]$, $\M$ and $W$ have a common ground $\overline{W}$.
$\overline{W} \subseteq \M \subseteq V$, hence $\overline{W}$ is a ground of $V$,
so $\M \subseteq W$.

(c) $\Rightarrow$ (a). 
For counting the grounds of $V[G]$, we can take a poset $\bbQ \in \M$ and 
an $(\M, \bbQ)$-generic $H$ such that $V=\M[H]$.
Since 
$\M$ is a solid bedrock of $V[G]$, $\M \subseteq W \subseteq V[G]$ for every ground $W$
of $V[G]$.
Hence by Fact \ref{2.31}, 
every ground $W$ of $V[G]$ is of the form $W[H']$ for some $H' \subseteq \bbQ' \subseteq ro(\bbQ)^{\M}$.
This shows that $V[G]$ has only  $(2^{\size{ro(\bbQ)^\M}})^{V[G]}$ many grounds.
\end{proof}

% We also note  the following, which was suggested by Yasuo Yoshinobu.
% \begin{lemma}
% $\M$ is a ground of $V$ if and only if
% there is a cardinal $\ka$
% such that $\M$ satisfies the $\ka$-covering and the $\ka$-approximation properties
% for $V$.
% \end{lemma}
% \begin{proof}
% If $\M$ is a ground of $V$,
% then by Fact \ref{3}, there is some $\ka$ such that
% $\M$ satisfies the $\ka$-covering and the $\ka$-approximation properties for $V$.

% Now suppose that $\M$ satisfies the $\ka$-covering and the $\ka$-approximation
% properties for $V$ for some $\ka$.
% By the strong DDG,
% there is a ground $W$ of $V$ such that
% $\p(\ka)^{\M}=\p(\ka)^W$.

% We see that $W=\M$.
% The inclusion $\M \subseteq W$ is trivial.
% For the converse, it is sufficient to show that if $A \in W$ is a set of ordinals,
% then $A \in \M$.
% First we treat with the case $\size{A}<\ka$.
% By the $\ka$-covering property of $\M$,
% there is a set $B \in \M$ of ordinals  with $A \subseteq B$ and $\size{B}<\ka$.
% We know $B \in \M \cap W$.
% Let $\pi:B \to  \ot(B)$ be the transitive collapsing map.
% Note that $\pi \in \M \cap W$ and $\ot(B)<\ka$.
% Put $A'=\pi``A$. We have $A' \in W$  since $A, B, \pi \in W$.
% Because $\p(\ka)^{\M}=\p(\ka)^W$,
% we have $A' \in \M$.
% Then $A=\pi^{-1}``A' \in \M$.

% We know $[ON]^{<\ka} \cap \M=[ON]^{<\ka} \cap W$.
% Then, by the $\ka$-approximation property of $\M$,
% it is easy to see that if $A \in W$ is a set of ordinals,
% then $A \in \M$.
% \end{proof}

The \emph{Ground Axiom} is the assertion that
$V$ has no proper ground (Reitz \cite{Reitz}).
If $\M$ is a ground of $V$, then it is clear that
$\M$ has no proper ground. Hence, 
if $\M$ is a ground of $V$,
then $\M$ is a natural model of ZFC+the Ground Axiom.

\section{Hyper-huge cardinals}

In this section, we prove that if a hyper-huge cardinal exists,
then $V$ has only set many grounds, so $\M$ is a ground of $V$.
To do this, we prove the following proposition which is interesting in its own right.
\begin{prop}\label{3.1.1}
Suppose $\ka$ is hyper-huge.
Let $W$ be a ground of $V$.
Then there is a poset $\bbP \in W$ and a $(W, \bbP)$-generic $G$ such that
$\size{\bbP}<\ka$ and $V=W[G]$.
\end{prop}
\begin{proof}
Fix a poset $\bbQ \in W$ and a $(W, \bbQ)$-generic $H$ such that
$V=W[H]$.

Fix an arbitrary inaccessible cardinal $\la>\ka$ with $\bbQ, H \in V_\la$.
Since $\ka$ is hyper-huge, we can find an elementary embedding $j:V \to N$ such that
the critical point $\ka$, $\la<j(\ka)$, and $N$ is closed under $j(\la)$-sequences.
Let $j(W)=\bigcup_{\alpha} j(W_\alpha)$. 
$j(W)$ is a transitive model of ZFC. Note that $j(W)_{j(\alpha)}=j(W_\alpha)$ for every $\alpha$.

First we show that \[
W_{j(\la)} \subseteq j(W_\la) \subseteq N_{j(\la)} =V_{j(\la)}.
\]
The equality $N_{j(\la)} =V_{j(\la)}$ follows from the closure property of $N$,
and the inclusion $j(W_\la) \subseteq N_{j(\la)}$ follows from the elementarity of $j$.
So the problem is the inclusion $W_{j(\la)} \subseteq j(W_\la)$.
Since $j(\la)$ is inaccessible, both 
$W_{j(\la)}$ and $j(W_\la)$ are models of ZFC.
Thus it is enough to check that if $x \in W_{j(\la)}$ is a set of ordinals,
then $x \in j(W_\la)$.

We use a well-known argument using disjoint stationary subsets, which is due to Solovay.
In $W$, fix a pairwise disjoint sequence $\seq{S_\alpha:\alpha<j(\la)}$ such that
$S_\alpha \subseteq j(\la) \cap \cof(\om)^W$ and $S_\alpha$ is stationary in $j(\la)$ in $W$.
Since $V=W[G]$ and $\size{\bbQ}<\la<j(\la)$,
each $S_\alpha$ is still stationary in $j(\la)$ in $V$.
Let $j(\seq{S_\alpha:\alpha<j(\la)})=\seq{S^*_\alpha:\alpha<j(j(\la))}$.

\begin{claim}
For every $\alpha<\sup(j``j(\la))$,
$\alpha \in j``j(\la)$ if and only if 
$S^*_\alpha \cap \sup(j``j(\la))$ is stationary in $sup(j``j(\la))$ in $j(W)$.
\end{claim}
\begin{proof}[Proof of Claim]
First suppose $\alpha \in j``j(\la)$.
Take $\beta<j(\la)$ with $\alpha=j(\beta)$. We know $S^*_\alpha=j(S_\beta)$.
We check that $S^*_\alpha \cap \sup(j``j(\la))$ is stationary in $\sup(j``j(\la))$ in $j(W)$.
To do this, take a club $C \subseteq \sup(j``j(\la))$ with $C \in j(W)$.
Because the critical point of $j$ is uncountable, we know that
$j``j(\la)$ is an $\om$-club in $\sup(j``j(\la))$ in $V$,
that is, it is unbounded in $\sup(j``j(\la))$ and, 
for every $\xi \in \lim(\sup(j``j(\la)))$, if $\cf(\xi)=\om$ in $V$ then $\xi \in j``j(\la)$.
Hence we know that $C \cap j``j(\la)$ is an $\om$-club in $V$.
Put $D=j^{-1}``(C \cap j``j(\la))$. 
Again, since the critical point of $j$ is uncountable, $D$ is an $\om$-club in $j(\la)$.
$S_\beta$ is stationary in $j(\la)$ in $V$ and $\cf(\xi)=\om$ for every $\xi \in S_\beta$,
so we have that $S_\beta \cap D \neq \emptyset$.
Fix $\xi \in S_\beta \cap D$.
Then $j(\xi) \in j(S_\beta) \cap C$, so $j(\xi) \in S^*_\alpha \cap C$.

For the converse, suppose $S^*_\alpha \cap \sup(j``j(\la))$ is stationary in $j(W)$.
By the elementarity of $j$, $N$ is a forcing extension of $j(W)$ via poset $j(\bbQ)$.
We know $\size{j(\bbQ)}<j(\la)$.
Hence, if $E \in j(W)$ is stationary in $j(\la)$ in $j(W)$,
then $E$ remains stationary in $j(\la)$ in $N$.
This means that $S^*_\alpha \cap \sup(j``j(\la))$ remains stationary in $N$.
$\cf(\xi)=\om$ for every $\xi \in S^*_\alpha$ and $j``j(\la)$ is an $\om$-club,
hence we have $S^*_\alpha \cap j``j(\la) \neq \emptyset$.
Take $j(\xi) \in S^*_\alpha \cap j``j(\la)$.
By the elementarity of $j$, there is some $\beta<\la$ such that
$\xi \in S_\beta$. Hence $j(\xi) \in S^*_{j(\beta)} \cap S^*_\alpha$,
and then $j(\beta)=\alpha$ because the sequence $\seq{S^*_\gamma:\gamma<j(j(\la))}$ is pairwise disjoint.
\qedhere[Claim]
\end{proof}

By the claim,
we have that
\[
j``j(\la)=\{\alpha<\sup(j``j(\la)): S^*_\alpha \cap \sup(j``j(\la))\text{ is stationary in 
$\sup(j``j(\la))$ in $j(W)$} \}.
\]
Thus we have that $j``j(\la)$ is definable in $j(W)$, and $j``j(\la) \in j(W)$.
Now we see that, if $x \in W_{j(\la)}$ is a set of ordinals, then $x \in j(W_\la)$.
Clearly $j(x)  \in j(W)$. Since $x \subseteq j(\la)$,
we have that $j``x=j(x) \cap j``j(\la) \in j(W)$.
Because $j``j(\la) \in j(W)$, we have that $j \restriction j(\la) \in j(W)$,
thus $x=j^{-1}``(j``x) \in j(W)$. Again, since $x \subseteq j(\la)$,
we have $x \in j(W)_{j(\la)}=j(W_\la)$.

We  know $W_{j(\la)} \subseteq j(W_\la) \subseteq N_{j(\la)} =V_{j(\la)}$.
$V=W[H]$, hence $V_{j(\la)}=W_{j(\la)}[H]$, and
$W_{j(\la)} \subseteq j(W_\la)=N_{j(\la)}=V_{j(\la)}=W_{j(\la)}[H]$.
By Fact \ref{2.31},
there is a poset $\bbQ' \subseteq ro(\bbQ)$ and a $(j(W_\la), \bbQ')$-generic $H'$ such that
$j(W_\la)[H']=V_{j(\la)}=N_{j(\la)}$.
Since $\size{\bbQ}<\la<j(\ka)$, we have that
$\size{\bbQ'}<j(\ka)$. 
We may assume that $\bbQ' \in j(W_\la)_{j(\ka)}$ ($=j(W_\ka)$).
Therefore, $N$ satisfies the following statement:
\begin{quote}
There exists a poset $\bbQ'$ and a $(j(W_\la), \bbQ')$-generic $H$
such that $\bbQ' \in j(W_\ka)$ and $j(W_\la)[H']=N_{j(\la)}$.
\end{quote}
By the elementarity of $j$, the following holds in $V$:
\begin{quote}
There exists a poset $\bbQ'$ and a $(W_\la, \bbP)$-generic $H'$
such that $\bbQ' \in W_\ka$ and $W_\la[H']=V_{\la}$.
\end{quote}

Thus, for every sufficiently large inaccessible cardinal $\la$,
there exists a poset $\bbP$ and a $(W_\la, \bbP)$-generic $G$
such that $\bbP \in W_\ka$ and $W_\la[G]=V_{\la}$.
Since $\ka$ is hyper-huge, there are proper class many inaccessible cardinals.
Hence there must be a poset $\bbP \in V_\ka$ and a filter $G \subseteq \bbP$
such that
the collection
\[
\{\la:\text{$\la$ is inaccessible, $G$ is $(W_\la, \bbP)$-generic, and $W_\la[G]=V_\la$}\}
\]
forms a proper class.
Then it is clear that $G$ is $(W, \bbP)$-generic and $W[G]=V$.
\end{proof}

Now the following is immediate from Proposition \ref{3.1.1}:
\begin{cor}\label{3.3.1}
If $\ka$ is hyper-huge, then $V$ has only set many grounds.
\end{cor}
\begin{proof}
The assertion follows from Fact \ref{3} and Proposition \ref{3.1.1}.
For each ground $W$ of $V$,
one can find a poset $\bbP_W \in V_\ka$ and a $(W, \bbP)$-generic $G_W$
such that $W[G_W]=V$.
By Lemma \ref{3.5}, for two grounds $W$ and $W'$,
if $\bbP_W=\bbP_{W'}$, $\la=\size{\bbP_W}^W
=\size{\bbP_{W'}}^{W'}$,
 and $\p(\la)^{W}=\p(\la)^{W'}$
then $W=W'$.
Thus the assignment $W \mapsto \seq{\bbP_W, \size{\bbP}^W, \p(\size{\bbP}^W)^W}$
is an injection from the grounds to
$V_\ka$. This means that $V$ has only $\ka$ many grounds.
\end{proof}

By Corollary \ref{4.111}, Proposition \ref{3.1.1},
and Corollary \ref{3.3.1},
we have the following:
\begin{cor}\label{6.4.1}
Suppose $\ka$ is a hyper-huge cardinal. Then the following hold:
\begin{enumerate}
\item $\M$ is a solid bedrock of all forcing extensions of $V$.
\item There is a poset $\bbP\in \M_\ka$ and
a $(\M, \bbP)$-generic $G$ with $V=\M[G]$.
\item $V$ has only $<\ka$ many grounds.
\end{enumerate}
\end{cor}
\begin{proof}
(1) and (2) are immediate from 
Proposition \ref{3.1.1}, Corollaries \ref{4.111}, and \ref{3.3.1}.

For (3), 
fix a poset $\bbP \in \M_\ka$ and
a $(\M,\bbP)$-generic $G$ with $V=\M[G]$.
%By Proposition \ref{3.1.1}, we may assume that
%$\size{\bbP}<\ka$.
By Fact \ref{3},
every ground $W$ of $V$ is of the form
$\M[H]$ for some $H \subseteq \bbQ \subseteq ro(\bbP)^{\M}$.
Hence there are only $2^{\size{ro(\bbP)}}$ ($<\ka$) many grounds of $V$.
\end{proof}

As mentioned before,
$\M$ is a model of ZFC+the Ground Axiom if $\M$ is a ground.
If $\ka$ is hyper-huge, then $\M$ is a ground of $V$ and $V$ is a small forcing extension of
$V$ relative to $\ka$. Then one can check $\ka$ remains hyper-huge in $\M$ as follows:
Suppose $V=\M[G]$ for some $G \subseteq \bbP \in \M$ with $\size{\bbP}<\ka$.
For a given $\la \ge \ka$,
let $j:V\to N$ be an elementary embedding with critical point $\ka$ and ${}^{j(\la)} N \subseteq N$.
$j \restriction \M$ is an elementary embedding from $\M$ to $\M^{N}$.
$\M$ satisfies the $\mu$-covering and the $\mu$-approximation properties for $V$ for some $\mu<\ka$
by Fact \ref{3}. By the results in Hamkins \cite{Hamkins}, we know that
$j \restriction \M$ and $\M^N$ are definable in $\M$, and
$\M^N$ is closed under $j(\la)$-sequences in $\M$. $\la$ was arbitrary,
hence we have that $\ka$ is hyper-huge in $\M$.
Thus $\M$ is a model of ZFC+the Ground Axiom+hyper-huge cardinal exists.
\\

Proposition \ref{3.1.1} also shows the destructibility phenomenon of hyper-huge cardinals.
Laver \cite{Laver} proved  that a supercompact cardinal $\ka$ can be 
indestructible for $<\ka$-directed closed forcings.
In contrast to Laver's theorem,
Bagaria-Hamkins-Tsaprounis-Usuba \cite{BHTU} showed that,
if $\ka$ is extendible, then every non-trivial $<\ka$-closed forcing
must destroy the extendability of $\ka$.
Now, by Proposition \ref{3.1.1},
we know that  any hyper-huge cardinals must be destroyed by any non-trivial non-small forcings.
The following is just a rephrasing of Proposition \ref{3.1.1}:
\begin{cor}
Let $\ka$ be an infinite cardinal.
Let $\bbP$ be a poset, and suppose for every $p \in \bbP$,
the suborder $\{q \in \bbP:q \le p\}$ is not forcing equivalent to
a poset of size $<\ka$.
Then $\bbP$ forces that ``$\ka$ is not hyper-huge''.
\end{cor}

\section{Some consequences of main theorems}

In this section, we discuss some consequences of
our main theorems.
% in this paper.
% set-theoretic geology and generic multiverse.

The \emph{generic HOD}, $\gHOD$, is the intersection of
HOD of all forcing extensions of $V$
(Fuchs-Hamkins-Reitz \cite{FHR}).
$\gHOD$ is definable as the class $\bigcap\{\HOD^{V^{\col(\alpha)}}:\alpha$ is an ordinal$\}$.

%Now we can compute $\gHOD$ by the following simple way:
\begin{prop}\label{7.1.1}
$\HOD^{\M}$ is a subclass of $\gHOD$.
If $\M$ is a ground of $V$
then $\HOD^{\M}$ coincides with $\gHOD$.
\end{prop}
\begin{proof}
First note that, by Corollary \ref{4.111},
we know that $g\M=\M$,
hence $\HOD^{g\M}=\HOD^{\M}$.

We show that $\HOD^{\M}=\HOD^{g\M} \subseteq \gHOD$.
Take an arbitrary generic extension $V[G]$ of $V$.
It is enough  to show that $\HOD^{g\M} \subseteq \HOD^{V[G]}$.
Since $g\M$ is a parameter free definable class in $V[G]$,
we have that $g\M_\alpha \in \mathrm{OD}^{V[G]}$ for every $\alpha$.
This means that $\HOD^{\M}=\HOD^{g\M} \subseteq \HOD^{V[G]}$.

Next suppose $\M$ is a ground of $V$.
Put $V=\M[G]$ for some $G \subseteq \bbP \in \M$.
Then we can find an ordinal $\theta$ greater than $\size{2^{\size{\bbP}}}^\M$,
an $(\M, \col(\theta))$-generic $H$, and
a $(V, \col(\theta))$-generic $H'$ such that
$V[H']=\M[H]$,
where $\col(\theta)$ is a standard forcing adding a surjection from $\om$ onto $\theta$.
$\col(\theta)$ is weakly homogeneous,
so we have that $\HOD^{V^{\col(\theta)}}=\HOD^{\M[H]} \subseteq \HOD^{\M}$.
$\gHOD \subseteq \HOD^{V^{\col(\theta)}}$, thus we have $\gHOD \subseteq \HOD^{V^{\col(\theta)}} \subseteq \HOD^{\M}$.
\end{proof}
\begin{note}
If a hyper-huge cardinal exists,
then $\M$ is a ground of $V$ by Corollary \ref{6.4.1}.
Combining this observation with Proposition \ref{7.1.1},
we have that, if a hyper-huge cardinal exists, then
$\HOD^{\M}$ must coincide with $\gHOD$.
\end{note}
\begin{note}
We also note that 
$\HOD^{\M} \subsetneq \gHOD$ is possible.
We use the following fact:
\begin{fact}[Fuchs-Hamkins-Reitz \cite{FHR}]
There is a class forcing extension $V[G]$
in which $V=\M^{V[G]}=g\M^{V[G]}=\gHOD^{V[G]}$.
\end{fact}
Starting with a model $\HOD^V \subsetneq V$,
take a class forcing extension $V[G]$
in which $V=\M^{V[G]}=\gHOD^{V[G]}$.
Then $\HOD^{\M^{V[G]}}=\HOD^V \subsetneq V=\gHOD^{V[G]}$.
\end{note}

% We know that the grounds are definable uniformly,
% but its definition is not absolute between $V$ and its forcing extensions.
% Under hyper huge cardinal assumption, we have an absolute definition of the
% grounds:
% \begin{lemma}
% Suppose hyper huge cardinal exists.
% Then every ground of $V$ is of the form $g\M[G]$ for some $G \subseteq \bbP \in g\M$.
% Hence the collection
% \[
% \{g\M[G]:\text{$G$ is $(g\M, \bbP)$-generic for some $\bbP \in g\M$ and
% $g\M[G]$ satisfies the axiom of choice}\}
% \]
% is the grounds of $V$.
% \end{lemma}

% Fuchs-Hamkins-Reitz \cite{FHR}
% showed that every universe can be 
% the mantle of some outer model:

% \begin{fact}[Fuchs-Hamkins-Reitz \cite{FHR}]
% There exists a class forcing extension $V[G]$
% in which $V=\M^{V[G]}=g\M^{V[G]}$.
% \end{fact}
The following gives some restrictions on
 $V[G]$ in the above fact.

\begin{prop}
If $V \neq \M$,
then, for every outer transitive model $W \supseteq V$ of ZFC,
if 
$\M^W$ is a ground of $W$,
then $\M^{W}  \neq V$ (so $g\M^W \neq V)$.
In particular, if $W$ has a hyper-huge cardinal,
then $\M^W \neq V$.
\end{prop}
\begin{proof}
Take an outer model $W \supseteq V$.
By our assumption, $\M^W$ is a solid bedrock of $W$ by Corollary \ref{4.111}.
Since $V \neq \M^V$, we have that $V$ has a proper ground $M$.
If $V=\M^{W}$, then $M$ is a ground of $\M^{W}$,
so $M=\M^{W}=V$ by the minimality of $\M^{W}$.
This is a contradiction.
\end{proof}

Next we turn to the generic multiverse,
which was first introduced in
Woodin \cite{Woodin}.
% Roughly speaking,
% the generic multiverse is 
% a collection $\calF$ of transitive models of ZFC
% such that $\calF$ is closed under taking
% grounds and forcing extensions.
In \cite{Woodin}, Woodin defined his generic multiverse as follows.
For a given countable transitive model $M_0$ of ZFC,
the generic multiverse of $M_0$ is a collection $\calF$ of countable transitive models of ZFC
such that:
\begin{enumerate}
\item $M_0 \in \calF$.
\item If $M \in \calF$ and $W \subseteq M$ is a ground of $M$, then $W \in \calF$.
\item If $M \in \calF$, $\bbP \in M$ is a poset, and $G$ is $(M, \bbP)$-generic, then $M[G] \in \calF$.
\item $\calF$ is a minimal collection satisfying (1)--(3).
\end{enumerate}
Note that (4) is equivalent to the following (4'):
\begin{enumerate}
\item[(4')] For every $M, N \in \calF$,
there are finitely many $W_0,\dotsc, W_n \in \calF$ such that
$W_0=M$, $W_n=N$, and, for each $i<n$,
$W_i$ is a ground or a forcing extension of $W_{i+1}$.
\end{enumerate}
Note also that if $M_0$ is fixed, then Woodin's generic multiverse is unique for $M_0$,
so it is \emph{the} generic multiverse of $M_0$.

On the other hand, Steel \cite{Steel} gave a  slightly different definition of the generic multiverse.
For a transitive set model $M_0$ of ZFC,
a generic multiverse of $M_0$ is a collection $\calF$ of transitive models of ZFC
such that:
\begin{enumerate}
\item $M_0 \in \calF$.
\item If $M \in \calF$ and $W \subseteq M$ is a ground of $M$, then $W \in \calF$.
\item If $M \in \calF$ and $\bbP \in M$ is a poset, then there is an $(M, \bbP)$-generic  $G$
with $M[G] \in \calF$.
\item If $M, N \in \calF$,
then there is $W \in \calF$ which is a common forcing extension of $M$ and $N$.
\end{enumerate}
Unlike Woodin's generic multiverse, 
Steel's generic multiverse is not unique for $M_0$;
it is possible that $M_0$ has many Steel's generic multiverses.

To treat  Woodin's and Steel's generic multiverses simultaneously,
we adopt the following definition of generic multiverse, which is weaker than Woodin's and Steel's definitions.
\begin{define}\label{def. of multiverse}
For a transitive set model $M_0$ of ZFC,
a \emph{generic multiverse} of $M_0$
is a collection $\calF$ of transitive models of ZFC such that:
\begin{enumerate}
\item $M_0 \in \calF$.
\item If $M \in \calF$ and $W \subseteq M$ is a ground of $M$,
then $W \in \calF$.
\item If $M \in \calF$ and $\bbP \in M$ is a poset,
then there is an $(M,\bbP)$-generic $G$ with $M[G] \in \calF$.
\item For every $M, N \in \calF$,
there are finitely many $W_0,\dotsc, W_n \in \calF$ such that
$W_0=M$, $W_n=N$, and, for each $i<n$,
$W_i$ is a ground or a forcing extension of $W_{i+1}$.
\end{enumerate}
\end{define}
Note that if $M_0$ has a generic multiverse in our sense, then $M_0$ must be  countable or
$M_0\cap ON=\om_1$.
Our definition of generic multiverse is weak, 
so if $\calF$ is a generic multiverse in the sense
of Woodin or Steel, then it also is in our sense.
% in the definition.

% makes sense even if $M_0$ is a proper class or $M_0=V$.
% See Section 3 in \cite{FHR} how to formalize our generic multiverse when $M_0$ is a proper class.

Now fix a transitive model $M_0$ of ZFC with $M_0 \cap ON \le \om_1$.

\begin{lemma}\label{DDG of multiverse}
If $\calF$ is a generic multiverse of $M_0$,
then every two members of $\calF$ have a common ground.
\end{lemma}
\begin{proof}
Take $M, N \in \calF$.
By the definition of $\calF$,
there are finitely many $W_0,\dotsc, W_n \in \calF$ such that
$W_0=M$, $W_n=N$, and, for each $i<n$,
$W_i$ is a ground or a forcing extension of $W_{i+1}$.
By induction on $i <n$,
we show that $W_0$ and $W_{i+1}$ have a common ground.
The case $i=0$ is clear.
Suppose $W_{0}$ and $W_{i}$ have a common ground $W$.
If $W_i$ is  a ground of $W_{i+1}$, then $W$ is  a common ground of $W_0$ and $W_{i+1}$.
Suppose $W_{i}$ is a forcing extension of $W_{i+1}$.
Then $W$ and $W_{i+1}$ are grounds of $W_i$.
By the DDG of $W_{i}$, $W$ and $W_{i+1}$ have a common ground $W'$,
which is a common ground of $W_0$ and $W_{i+1}$.
\end{proof}
This lemma shows that, in Definition \ref{def. of multiverse},
under the presence of (1)--(3), 
the clause (4) can be replaced by
``every two members of $\calF$ have a common ground''.

\begin{prop}\label{multiverse}
Let $\calF$ be a generic multiverse of $M_0$.
Then $\calF$ is just the collection \begin{center}
$\{M \in \calF: M$ is a forcing extension of 
some ground of $M_0\}$.
\end{center}
Moreover, for $M, N \in \calF$, if $M \subseteq N$,
 then $M$ is a ground of $N$.
\end{prop}
\begin{proof}
Let $\calF'=
\{M \in \calF: M$ is a forcing extension of 
some ground of $M_0\}$.
It is clear that $\calF' \subseteq \calF$.
If $M \in \calF$, then $M$ and $M_0$ have a common ground $W \in \calF$ by
Lemma \ref{DDG of multiverse}.
Then $M \in \calF'$ by the definition.

Next take $M, N \in \calF$ with $M \subseteq N$.
By Lemma \ref{DDG of multiverse} again,
$M$ and $N$ have a common ground $W$.
Since $N$ is a forcing extension of $W$ and
$W \subseteq M \subseteq N$,
$N$ is a forcing extension of $M$ by Fact \ref{2.31}.
% $\calF$ is closed under taking forcing extensions and grounds.
%
% By the definition of $\calF$,
% it is clear that $\calF$ is closed under taking forcing extensions.
% For taking the grounds,
% take $M \in \calF$, and let $W \subseteq M$ be a ground of $M$.
% By the definition of $\calF$,
% there is a ground $W'$ of $V$ such that $W'$ is a ground of $M$.
% Now $W'$ and $W$ are grounds of $M$,
% hence $W'$ and $W$ have a common ground $\overline{W}$ by DDG.
% $\overline{W} \subseteq W' \subseteq V$, hence $\overline{W}$ is a ground of $V$.
% $W$ is a forcing extension of a ground $\overline{W}$ of $V$,
% thus we have $W \in \calF$.
%
% Next, for  (1), if $M$ and $N$ are members of the generic multiverse of $V$,
% then there are grounds $W_0$ and $W_1$ of $V$ such that
% $W_0$ is a ground of $M$ and $W_1$ is of $N$.
% By DDG, $W_0$ and $W_1$ have a common ground $W$,
% then $W$ is a common ground of $M$ and $N$.
%
% (2) is immediate from (1).
\end{proof}

If $\M$ is a ground,
then every ground is a forcing extension of $\M$.
Hence we have the following simple view of a generic multiverse under
this assumption:
% The following gives a simple view of the generic multiverse
% under large cardinal assumption.

\begin{prop}\label{5.7.1}
Suppose 
$\M^{M_0}$ is a ground of $M_0$.
Then for every generic multiverse $\calF$ of $M_0$,
the following hold:
\begin{enumerate}
\item $\M^{M_0}=\M^{N}$ for every $N \in \calF$.
\item $\M^{M_0}$ is the  minimum member of $\calF$.
\item Every member of $\calF$ is a forcing extension of $\M^{M_0}$.
\end{enumerate}
% and the generic multiverse of $V$ is exactly the same to
% the collection of all forcing extensions of $\M$.
\end{prop}

Let $\calF$ be a generic multiverse of $M_0$.
A sentence $\varphi$ of set-theory is
a \emph{multiverse truth} (on $\calF$) if every member of a generic multiverse $\calF$ satisfies $\varphi$.
Woodin \cite{Woodin} showed that there is a computable translation $()^*$ on sentences such that,
for every sentence $\varphi$,
$\varphi$ is a multiverse truth if and only if $(\varphi)^*$ holds in $M_0$.
By the DDG, we have the following simple translation.
The next proposition is immediate from Propositions \ref{multiverse} and \ref{5.7.1}.
\begin{prop}\label{5.7.2}
%Let $M_0$ be a transitive set model of ZFC,
%and 
Let $\calF$ be a generic multiverse of $M_0$, and
$\varphi$  a sentence of set theory.
\begin{enumerate}
\item $\varphi$ is a multiverse truth if and only if,
for every ground $W$ of $M_0$ and every poset $\bbP \in W$,
$\bbP$ forces $\varphi$ in $W$.
\item If 
$\M^{M_0}$ is a ground of $M_0$,
then $\varphi$ is a multiverse truth
if and only if,
for every poset $\bbP \in \M^{M_0}$,
$\bbP$ forces $\varphi$ in $\M^{M_0}$.
\end{enumerate}
\end{prop}

We also have the following:
\begin{prop}
Let $\calF$ be a generic multiverse of $M_0$, and
$\varphi$  a sentence of set theory.
Then $\varphi$ is a multiverse truth if and only if,
for every poset $\bbP \in M_0$,
$\bbP$ forces ``every ground satisfies $\varphi$'' in $M_0$.
\end{prop}
\begin{proof}
By the definition of a generic multiverse,
if $\varphi$ is a multiverse truth then
every poset forces ``every ground satisfies $\varphi$'' in $M_0$.

For the converse, suppose
there is  $M \in \calF$ such that $\varphi$ fails in $M$.
By Lemma \ref{DDG of multiverse},
$M$ and $M_0$ have a common ground $N \in \calF$.
Since $M$ is a forcing extension of $N$,
there is a poset $\bbP \in N$ which forces
$\neg \varphi$ in $N$.
Fix a large $\theta>2^{\size{\bbP}}$.
We can find an $(M_0, \mathrm{Col}(\theta))$-generic $G$
with $M_0[G]\in \calF$.
In $M_0[G]$, there is an $(N, \bbP)$-generic $H \in M_0[G]$.
Then $N \subseteq N[H] \subseteq M_0[G]$.
$\varphi$ fails in $N[H]$.
Since $N$ is a ground of $M_0$ and $M_0$ is a ground of $M_0[G]$,
$N[H]$ is a ground of $M_0[G]$.
This means that 
$\mathrm{Col}(\theta)$ forces that
``there is a ground which does not satisfy $\varphi$'' in $M_0$.
\end{proof}

\begin{note}
If a hyper-huge cardinal exists,
then $\M$ is a ground of $V$.
Hence, if a hyper-huge cardinal exists in $M_0$,
then the conclusions of Proposition \ref{5.7.1} and of (2) in Proposition \ref{5.7.2}  hold.
\end{note}

Finally we discuss the maximality principles,
that were studied by Hamkins \cite{Hamkins2}.
Let $\varphi(x)$ be a formula of set theory with free variable $x$,
and $s$ a set.
A sentence $\varphi(s)$ is \emph{forceable} if
there is a poset $\bbP$ which forces $\varphi(s)$,
and $\varphi(s)$ is \emph{necessary} if
every poset $\bbP$ forces $\varphi(s)$.
The \emph{Maximality principle}, $\mathrm{MP}$,
is the assertion that,
every forceably necessary sentence is true in $V$,
that is, for every sentence $\varphi$, 
if
\[
\exists \bbP\,( \Vdash_\bbP\text{``}\forall \bbQ\, \Vdash_\bbQ \varphi\text{''}),
\]
then $\varphi$ holds in $V$.
$\mathrm{MP}(\bbR)$ is the assertion that, 
for every formula $\varphi(x)$ and $r \in \bbR$,
if $\varphi(r)$ is forceably necessary,
then $\varphi(r)$ holds in $V$.
$\square \mathrm{MP}(\bbR)$ is the assertion that
$\mathrm{MP}(\bbR)$ is necessary,
that is, every poset $\bbP$ forces that $\mathrm{MP}(\bbR^{V^{\bbP}})$ holds.

Hamkins \cite{Hamkins2} showed that 
$\mathrm{MP}$ is equiconsistent with ZFC,
and that $\mathrm{MP}(\bbR)$ is  consistent relative to some large cardinal assumption
which is weaker than a Mahlo cardinal.
Moreover, by the proofs in \cite{Hamkins2},
$\mathrm{MP}$ and even $\mathrm{MP}(\bbR)$ are consistent with
almost all large cardinals.
Woodin proved that 
$\square \mathrm{MP}(\bbR)$ is consistent relative to some large cardinal assumption.
However, Woodin's model is an extension of some canonical model of
$\mathrm{AD}_{\bbR}$,
and  no strong large cardinals exist in the resulting model.
Hence it is natural to ask if 
$\square \mathrm{MP}(\bbR)$ is consistent with strong large cardinals,
for example, supercompact cardinals.
Now we can show that our hyper-huge cardinal is inconsistent with
$\square \mathrm{MP}(\bbR)$.
\begin{prop}\label{14}
If $\square \mathrm{MP}(\bbR)$ holds,
then $\M$ is not a ground of $V$.
In particular, if there exists a hyper-huge cardinal,
then $\square \mathrm{MP}(\bbR)$ fails.
\end{prop}
The proof is immediate from Corollary \ref{4.111}
and the following fact:

\begin{fact}[Hamkins \cite{Hamkins2}]\label{17}
Suppose $\square \mathrm{MP}(\bbR)$.
Let $M \subseteq V$ be a forcing invariant parameter free definable transitive model of ZFC.
Then, for every infinite ordinal $\alpha$,
$(\alpha^+)^{M}$ is not a cardinal in $V$.
In particular,
if $\square \mathrm{MP}(\bbR)$ holds,
then, for every infinite ordinal $\alpha$,
$(\alpha^+)^{\M}$ is not a cardinal in $V$.
\end{fact}

{\bf Acknowledgements}:
We would like to 
thank   Joel David Hamkins  and Daisuke Ikegami for
many fruitful discussions.
Some of the ideas that came out of those discussions were
used in this paper.
We also thank Paul Larson and Yo Matsubara,
they gave the author many corrections of the draft and insightful suggestions.
Finally we are grateful to David Aspero, Hiroshi Sakai,
Kostantinos Tsaprounis, Yasuo Yoshinobu for their useful comments.
This research was supported by KAKENHI grant Nos. 15K17587 and 15K04984.

\printindex
\end{document}